\documentclass[12pt,oneside,reqno]{amsart}

\usepackage{amsmath,amssymb,amsthm,textcomp}
\usepackage{amsfonts,graphicx}
\usepackage[mathscr]{eucal}
\usepackage{color}
\usepackage{csquotes}
\usepackage[backend=bibtex,%
firstinits=true,%
doi=false,%
isbn=true,%
url=false,%
maxnames=99]{biblatex}%

\AtEveryBibitem{\clearfield{issn}}
\AtEveryCitekey{\clearfield{issn}}
\numberwithin{equation}{section}
\DeclareNameAlias{sortname}{last-first}
\theoremstyle{definition}

\numberwithin{equation}{section}

\newcommand{\ncom}{\newcommand}

\ncom{\beq}{\begin{equation}}
\ncom{\eeq}{\end{equation}}
\ncom{\bea}{\begin{eqnarray*}}
\ncom{\eea}{\end{eqnarray*}}
\ncom{\beqa}{\begin{eqnarray}}
\ncom{\eeqa}{\end{eqnarray}}
\ncom{\nno}{\nonumber}
\ncom{\non}{\nonumber}
\ncom{\ds}{\displaystyle}
\ncom{\half}{\frac{1}{2}}
\ncom{\mbx}{\makebox{.25cm}}
\ncom{\hs}{\mbox{\hspace{.25cm}}}
\ncom{\rar}{\rightarrow}
\ncom{\Rar}{\Rightarrow}
\ncom{\noin}{\noindent}
\ncom{\bc}{\begin{center}}
\ncom{\ec}{\end{center}}
\ncom{\sz}{\scriptsize}
\ncom{\rf}{\ref}
\ncom{\s}{\sqrt{2}}
\ncom{\sgm}{\sigma}
\ncom{\Sgm}{\Sigma}
\ncom{\psgm}{\sigma^{\prime}}
\ncom{\dt}{\delta}
\ncom{\Dt}{\Delta}
\ncom{\lmd}{\lambda}
\ncom{\Lmd}{\Lambda}
\ncom{\Th}{\Theta}
\ncom{\e}{\eta}
\ncom{\eps}{\epsilon}
\ncom{\pcc}{\stackrel{P}{>}}
\ncom{\lp}{\stackrel{L_{p}}{>}}
\ncom{\dist}{{\rm\,dist}}
\ncom{\sspan}{{\rm\,span}}
\ncom{\re}{{\rm Re\,}}
\ncom{\im}{{\rm Im\,}}
\ncom{\sgn}{{\rm sgn\,}}
\ncom{\ba}{\begin{array}}
\ncom{\ea}{\end{array}}
\ncom{\hone}{\mbox{\hspace{1em}}}
\ncom{\htwo}{\mbox{\hspace{2em}}}
\ncom{\hthree}{\mbox{\hspace{3em}}}
\ncom{\hfour}{\mbox{\hspace{4em}}}
\ncom{\vone}{\vskip 2ex}
\ncom{\vtwo}{\vskip 4ex}
\ncom{\vonee}{\vskip 1.5ex}
\ncom{\vthree}{\vskip 6ex}
\ncom{\vfour}{\vspace*{8ex}}
\ncom{\norm}{\|\;\;\|}
\ncom{\integ}[4]{\int_{#1}^{#2}\,{#3}\,d{#4}}
\ncom{\vspan}[1]{{{\rm\,span}\{ #1 \}}}
\ncom{\dm}[1]{ {\displaystyle{#1} } }
\ncom{\ri}[1]{{#1} \index{#1}}

\newtheorem{theorem}{\bf Theorem}[section]
\newtheorem{remark}{\bf Remark}[section]

\newtheorem{lemma}{Lemma}[section]
\newtheorem{corollary}{Corollary}[section]

\newtheoremstyle
    {remarkstyle}
    {}
    {11pt}
    {}
    {}
    {\bfseries}
    {:}
    {     }
    {\thmname{#1} \thmnumber{#2} }

\theoremstyle{remarkstyle}

\def\eps{\varepsilon}

\begin{document}
\title{On Distributions of Certain State Dependent Fractional Point Processes}
\author[Kuldeep Kumar Kataria]{K. K. Kataria}
\address{Kuldeep Kumar Kataria, Department of Mathematics,
 Indian Institute of Technology Bombay, Powai, Mumbai 400076, INDIA.}
 \email{kulkat@math.iitb.ac.in}
\author{P. Vellaisamy}
\address{P. Vellaisamy, Department of Mathematics,
 Indian Institute of Technology Bombay, Powai, Mumbai 400076, INDIA.}
 \email{pv@math.iitb.ac.in}
\thanks{The research of K. K. Kataria was supported by IRCC, IIT Bombay.}
\date{January 04, 2018}
\begin{abstract}
We obtain the explicit expressions for the state probabilities of various state dependent fractional point processes recently introduced and studied by Garra {\it et al.} (2015). The inversion of the Laplace transforms of the state probabilities of such processes is rather cumbersome and involved. We employ the Adomian decomposition method to solve the difference differential equations governing the state probabilities of these state dependent processes. The distributions of some convolutions of the Mittag-Leffler random variables are derived as special cases of the obtained results.
\end{abstract}

\maketitle
\section{Introduction}
The Poisson process is an important counting process which has applications in several fields. A characterization of the Poisson process can be given in terms of Kolmogorov equations. A stochastic process $\{N(t,\lambda)\}_{t\geq0}$ is said to be a Poisson process with intensity $\lambda>0$ if the process has independent and stationary increments, and the state probabilities $p(n,t)=\mathrm{Pr}\{N(t,\lambda)=n\}$ satisfy 
\begin{equation}\label{tedft}
\frac{\mathrm{d}}{\mathrm{dt}}p(n,t)=-\lambda(p(n,t)-p(n-1,t)),\ \ n\geq 0,
\end{equation}
with initial condition $p(0,0)=1$, and $p(-1,t)=0$, $t\geq 0$. The conditions $p(0,0)=1$, $N(0,\lambda)=0$ a.s. and $p(n,0)=0$ for all $n\geq 1$ are essentially equivalent.

Recently, many authors introduced various fractional generalizations of the homogeneous Poisson process by replacing the time derivative in difference differential equations (\ref{tedft}) by various fractional derivatives such as Riemann-Liouville fractional derivative (see Laskin (2003)), Caputo fractional derivative (see Beghin and Orsingher (2009)), Prabhakar derivative (see Polito and Scalas (2016)), Saigo fractional derivative (see Kataria and Vellaisamy (2017)), {\it etc}. For space fractional generalization of the classical homogeneous Poisson process, we refer to Orsingher and Polito (2012), where space fractional Poisson process (SFPP), and space and time fractional Poisson process (STFPP) are introduced and studied. The distributions of these generalized Poisson processes are generally obtained by evaluating the respective Laplace transforms. However, in certain cases inversion of the Laplace transform is not at all simple. Some recently introduced state dependent fractional point processes by Garra {\it et al.} (2015) are illustrative examples.

Garra {\it et al.} (2015) studied three state dependent fractional point processes, where the difference differential equations governing the state probabilities of introduced processes depend on the number of events that have occurred till time $t\geq0$. 

The first version of the state dependent time fractional Poisson process (SDTFPP-I) $\{N_1(t,\lambda)\}_{t\geq0}$, $\lambda>0$, is defined as the stochastic process whose probability mass function (pmf) $p^{\alpha_n}(n,t)=\mathrm{Pr}\{N_1(t,\lambda)=n\}$, satisfies (see Eq. (1.1), Garra {\it et al.} (2015))
\begin{equation}\label{rc2newq}
\partial_t^{\alpha_n} p^{\alpha_n}(n,t)=-\lambda(p^{\alpha_n}(n,t)-p^{\alpha_{n-1}}(n-1,t)),\ \ 0<\alpha_n\leq 1,\ n\geq 0,
\end{equation}
with $p^{\alpha_{-1}}(-1,t)=0$, $t\geq 0$, and the initial conditions $p^{\alpha_n}(0,0)=1$ and $p^{\alpha_n}(n,0)=0$, $n\geq 1$. For each $n\geq0$, $\partial_t^{\alpha_n}$ denotes the fractional derivative in Caputo sense which is defined as
\begin{equation}\label{capt}
\partial_t^{{\alpha_n}}f(t):=\left\{
\begin{array}{ll}
\frac{1}{\Gamma{(1-{\alpha_n})}}\int^t_{0} (t-s)^{-{\alpha_n}}f'(s)\,\mathrm{d}s,\ \ 0<{\alpha_n}<1,\\\\
f'(t),\ \ {\alpha_n}=1.
\end{array}
\right.
\end{equation}
The order of the Caputo derivative in difference differential equations (\ref{rc2newq}) depends on the number of events till time $t$. The Laplace transform of the state probabilities of SDTFPP-I is given by
\begin{equation}\label{lt1}
\tilde{p}^{\alpha_n}(n,s)=\int_0^{\infty}p^{\alpha_n}(n,t)e^{-st}\,\mathrm{d}t=\frac{\lambda^ns^{\alpha_0-1}}{\prod_{k=0}^n(s^{\alpha_k}+\lambda)},\ \ s>0.
\end{equation}

The second version of the state dependent time fractional Poisson process (SDTFPP-II) $\{N_2(t,\lambda)\}_{t\geq0}$, $\lambda>0$, is defined as the stochastic process with independent but nonidentically distributed waiting times $W_n$ (see Section 3, Garra {\it et al.} (2015)) with $\mathrm{Pr}\{W_n>t\}=E_{\beta_n}(-\lambda t^{\beta_n})$, $0<\beta_n\leq 1$, where $E_{\beta_n}(.)$ is the Mittag-Leffler function defined by
\begin{equation}\label{fastt}
E_{\beta_n}(x)=\sum_{k=0}^{\infty‎}\frac{x^k}{\Gamma(k\beta_n+1)},\ x\in\mathbb{R}.
\end{equation}
Let $I_t^{\beta_n}$ denote the Riemann-Liouville (RL) fractional integral of order $\beta_n$, $n\geq0$, defined by
\begin{equation*}\label{rli}
	I^{\beta_n}_tf(t):=\frac{1}{\Gamma{({\beta_n})}}\int^t_{0} (t-s)^{{\beta_n}-1}f(s)\,\mathrm{d}s.
\end{equation*}
Note that $\partial^{\beta}_tf(t)=I^{1-\beta}_tf'(t)$, $0<\beta<1$. Garra {\it et al.} (2015) showed that the pmf $p^{\beta_n}(n,t)=\mathrm{Pr}\{N_2(t,\lambda)=n\}$ of SDTFPP-II satisfies
\begin{equation}\label{r2newq}
p^{\beta_n}(n,t)=p^{\beta_n}(n,0)-\lambda(I_t^{\beta_n}p^{\beta_n}(n,t)-I_t^{\beta_{n-1}}p^{\beta_{n-1}}(n-1,t)),\ \ n\geq 0,
\end{equation}
with $p^{\beta_{-1}}(-1,t)=0$, $t\geq 0$, and the initial conditions $p^{\beta_n}(0,0)=1$ and $p^{\beta_n}(n,0)=0$, $n\geq 1$. The Laplace transform of the state probabilities of SDTFPP-II is given by
\begin{equation}\label{lt2}
\tilde{p}^{\beta_n}(n,s)=\int_0^{\infty}p^{\beta_n}(n,t)e^{-st}\,\mathrm{d}t=\frac{\lambda^ns^{\beta_n-1}}{\prod_{k=0}^n(s^{\beta_k}+\lambda)},\ \ s>0.
\end{equation}

When $\alpha_{n}=\alpha$ and $\beta_{n}=\beta$ for all $n\geq0$, SDTFPP-I and SDTFPP-II reduce to the time fractional Poisson process (TFPP) studied by Beghin and Orsingher (2009). Further, the case $\alpha_{n}=1$ and $\beta_{n}=1$ for all $n\geq0$, gives the classical homogeneous Poisson process.

A fractional version of the classical nonlinear birth process, namely, fractional pure birth process (FPBP) was introduced by Orsingher and Polito (2010). Garra {\it et al.} (2015) studied a third fractional point process by introducing the state dependency in the difference differential equations governing the state probabilities of FPBP. The state dependent fractional pure birth process (SDFPBP) $\{\mathcal{N}(t,\lambda_n)\}_{t\geq0}$, $\lambda_n>0$, is defined as the stochastic process whose pmf $p^{\nu_n}(n,t)=\mathrm{Pr}\{\mathcal{N}(t,\lambda_n)=n\}$ satisfies (see Eq. (4.1), Garra {\it et al.} (2015))
\begin{equation}\label{gh22}
\partial_t^{\nu_n} p^{\nu_n}(n,t)=-\lambda_np^{\nu_n}(n,t)+\lambda_{n-1}p^{\nu_{n-1}}(n-1,t),\ \ 0<\nu_n\leq 1,\ n\geq 1,
\end{equation}
with $p^{\nu_{0}}(0,t)=0$, $t\geq 0$, and the initial conditions $p^{\nu_1}(1,0)=1$ and $p^{\nu_n}(n,0)=0$, $n\geq 2$. The Laplace transform of the state probabilities of SDFPBP is given by
\begin{equation}\label{lt3}
\tilde{p}^{\nu_n}(n,s)=\int_0^{\infty}p^{\nu_n}(n,t)e^{-st}\,\mathrm{d}t=\frac{s^{\nu_1-1}\prod_{k=1}^{n-1}\lambda_k}{\prod_{k=1}^n(s^{\nu_k}+\lambda_k)},\ \ s>0.
\end{equation}
More recently, the semi-Markovian nature of these state-dependent processes are studied in Orsingher {\it et al.} (2017), and Ricciuti and Toaldo (2017). They studied some semi-Markov processes and their connection with state-dependent models.

The state probabilities of SDTFPP-I, SDTFPP-II and SDFPBP can be obtained by inverting the Laplace transforms given by (\ref{lt1}), (\ref{lt2}) and (\ref{lt3}) respectively. However, the inversion of the Laplace transforms is rather involved and cumbersome, as mentioned in Garra {\it et al.} (2015). Only the explicit expressions for $p^{\alpha_{0}}(0,t)$ and $p^{\alpha_{1}}(1,t)$ in terms of the generalized Mittag-Leffler function were obtained by them for SDTFPP-I process.

In this paper, we obtain explicit expressions for the state probabilities $p^{\alpha_{n}}(n,t)$, $p^{\beta_{n}}(n,t)$ and $p^{\nu_{n}}(n,t)$ of SDTFPP-I, SDTFPP-II and SDFPBP for all $n$, by using the Adomian decomposition method (ADM). This method has an advantage over existing methods in obtaining the state probabilities of such processes. We also establish some additional results for these processes. In particular, the distributions of the convolutions of Mittag-Leffler random variables are obtained as special cases.

\section{Adomian decomposition method}
In this section, we briefly describe the mechanism of ADM. Consider a functional equation of the form
\begin{equation}\label{1.1}
u=f+L(u)+H(u),
\end{equation}
where $f$ is some known function, and $L$ and $H$ are linear and nonlinear operators respectively. In ADM (see Adomian (1986)), the solution of the above functional equation is expressed in the form of an infinite series
\begin{equation}\label{1.2}
u(x,t)=‎‎\sum_{n=0}^{\infty‎}u_n(x,t),
\end{equation}
and the nonlinear term $H(u)$ is assumed to decompose as
\begin{equation}\label{1.3}
H(u)=‎‎\sum_{n=0}^{\infty‎}A_n(u_0,u_1,\ldots,u_n),
\end{equation}
where $A_n$ denotes the $n$-th Adomian polynomial in $u_0,u_1,\ldots,u_n$. Also, the series (\ref{1.2}) and (\ref{1.3}) are assumed to be absolutely convergent. So, (\ref{1.1}) can be rewritten as
\begin{equation*}\label{1.4}
‎‎\sum_{n=0}^{\infty‎}‎u_n=f+‎‎\sum_{n=0}^{\infty‎}L(u_n)+‎‎\sum_{n=0}^{\infty‎}A_n(u_0,u_1,\ldots,u_n).
\end{equation*}
Thus, in accordance with ADM, $u_n$'s are obtained by the following recursive relations
\begin{equation}\label{1.5}
u_0=f\ \ \ \ \mathrm{and}\ \ \ \ u_n=L(u_{n-1})+A_{n-1}(u_0,u_1,\ldots,u_{n-1}),\ \ n\geq 1.
\end{equation}

The only difficult but crucial step involved in ADM is the computation of Adomian polynomials. For more details on these polynomials we refer to Rach (1984), Duan (2010) and Kataria and Vellaisamy (2016). In the absence of the nonlinear term $H(u)$, the ADM can be used effectively as the recursive relationship (\ref{1.5}) then simply reduces to $u_n=L(u_{n-1})$ with $u_0=f$.

Kataria and Vellaisamy (2017) obtained the distribution of Saigo space time fractional Poisson process (SSTFPP) via ADM, which was otherwise difficult to obtain using the prevalent method of inverting Laplace transform. As special case, the state probabilities of TFPP, SFPP and STFPP follow easily.

Note that the functional equations corresponding to the difference differential equations (\ref{rc2newq}), (\ref{r2newq}) and (\ref{gh22}) of SDTFPP-I, SDTFPP-II and SDFPBP respectively do not involve any nonlinear term. Hence, the ADM conveniently and rapidly gives the state probabilities as the series solutions of the corresponding difference differential equations.

\section{State dependent time fractional Poisson process-I}
The state probabilities $p^{\alpha_n}(n,t)$ of SDTFPP-I can be obtained by inverting its Laplace transform given by (\ref{lt1}). Garra {\it et al.} (2015) obtained the explicit expressions only for $p^{\alpha_0}(0,t)$ and $p^{\alpha_1}(1,t)$, as it's difficult to obtain the state probabilities for $n\geq2$ using the Laplace transform method.  Also, the distribution of SDTFPP-I can be expressed in terms of stable and inverse stable subordinators (see Eq. 2.12, Garra {\it et al.} (2015)). Here we obtain the explicit expressions for state probabilities of SDTFPP-I via ADM.

The following lemma will be used (see Eq. 2.1.16, Kilbas {\it et. al.} (2006)).
\begin{lemma}
	Let $I^\alpha_t$ be the RL fractional integral of order $\alpha>0$. Then for $\rho> -1$,
	\begin{equation*}
	I^\alpha_tt^{\rho}=\frac{\Gamma(\rho+1)}{\Gamma(\rho+\alpha+1)}t^{\rho+\alpha}.
	\end{equation*}
\end{lemma}
Also, the following holds for RL integral and the Caputo derivative (see Eq. 2.4.44, Kilbas {\it et. al.} (2006))
\begin{equation*}
I^\alpha_t\partial^\alpha_tf(t)=f(t)-f(0),\ \ 0<\alpha\leq 1.
\end{equation*}

In subsequent part of the paper, the set of nonnegative integers is denoted by $\mathbb{N}_0$.
\begin{theorem}\label{t1}
	Consider the following difference-differential equations governing the state probabilities of SDTFPP-I:
	\begin{equation}\label{r2new}
	\partial_t^{\alpha_n} p^{\alpha_n}(n,t)=-\lambda(p^{\alpha_n}(n,t)-p^{\alpha_{n-1}}(n-1,t)),\ \ 0<\alpha_n\leq 1,\ \lambda>0,\ n\geq 0,
	\end{equation}
	with $p^{\alpha_0}(0,0)=1$ and $p^{\alpha_n}(n,0)=0$, $n\geq 1$. The solution of (\ref{r2new}) is given by
	\begin{equation}\label{2.4kky}
	p^{\alpha_n}(n,t)=(-1)^n\sum_{k=n}^{\infty‎}(-\lambda)^k\underset{\Theta^k_n}{\sum}\frac{t^{\sum_{j=0}^nk_j\alpha_j}}{\Gamma(1+\sum_{j=0}^nk_j\alpha_j)},\ \ n\geq0,
	\end{equation}
	where 
	$
	\Theta^k_n=\{(k_0,k_1,\ldots,k_n):\ \sum_{j=0}^nk_j=k,\ k_0\in\mathbb{N}_0,\ k_j\in\mathbb{N}_0\backslash\{0\},\ 1\leq j\leq n\}.
	$
\end{theorem}
\begin{proof}
	Applying RL integral $I^{\alpha_n}_t$ on both sides of (\ref{r2new}), we get
	\begin{equation}\label{yth}
	p^{\alpha_n}(n,t)=p^{\alpha_n}(n,0)-\lambda I_t^{\alpha_n}(p^{\alpha_n}(n,t)-p^{\alpha_{n-1}}(n-1,t)),\ \ n\geq 0.
	\end{equation}
	
	Note that $p^{\alpha_{-1}}(-1,t)=0$ for $t\geq0$. For $n=0$, substitute $p^{\alpha_0}(0,t)=\sum_{k=0}^{\infty}p^{\alpha_0}_{k}(0,t)$ in (\ref{yth}) and apply ADM to get
	\begin{equation*}
	\sum_{k=0}^{\infty}p^{\alpha_0}_{k}(0,t)=p^{\alpha_0}(0,0)-\lambda\sum_{k=0}^{\infty} I_t^{\alpha_0} p^{\alpha_0}_{k}(0,t).
	\end{equation*}
	Thus, $p^{\alpha_0}_{0}(0,t)=p^{\alpha_0}(0,0)=1$ for all $t>0$ and $p^{\alpha_0}_{k}(0,t)=-\lambda I_t^{\alpha_0} p^{\alpha_0}_{k-1}(0,t)$, $k\geq 1$. Hence,
	\begin{equation*}
	p^{\alpha_0}_{1}(0,t)=-\lambda I_t^{\alpha_0} p^{\alpha_0}_{0}(0,t)=-\lambda I_t^{\alpha_0} t^0=\frac{-\lambda t^{\alpha_0}}{\Gamma(\alpha_0+1)},
	\end{equation*}
	and similarly 
	\begin{equation*}
	p^{\alpha_0}_{2}(0,t)=\frac{(-\lambda t^{\alpha_0})^2}{\Gamma(2\alpha_0+1)},\ \ p^{\alpha_0}_{3}(0,t)=\frac{(-\lambda t^{\alpha_0})^3}{\Gamma(3\alpha_0+1)},
	\end{equation*}
	and {\it etc}. Let
	\begin{equation}
	p^{\alpha_0}_{k-1}(0,t)=\frac{(-\lambda t^{\alpha_0})^{k-1}}{\Gamma((k-1){\alpha_0}+1)}.
	\end{equation}
	Then
	\begin{equation*}
	p^{\alpha_0}_{k}(0,t)=-\lambda I_t^{\alpha_0} p^{\alpha_0}_{k-1}(0,t)=\frac{(-\lambda)^{k}}{\Gamma((k-1){\alpha_0}+1)} I_t^{\alpha_0} t^{(k-1){\alpha_0}}=\frac{(-\lambda t^{\alpha_0})^{k}}{\Gamma(k\alpha_0+1)},\ \ k\geq 0.
	\end{equation*}
	Therefore,
	\begin{equation}
	p^{\alpha_0}(0,t)=\sum_{k=0}^{\infty}\frac{(-\lambda t^{\alpha_0})^{k}}{\Gamma(k{\alpha_0}+1)},
	\end{equation}
	and thus the result holds for $n=0$.
	
	For $n=1$, substituting $p^{\alpha_1}(1,t)=\sum_{k=0}^{\infty}p^{\alpha_1}_{k}(1,t)$ in (\ref{yth}) and applying ADM, we get
	\begin{equation*}
	\sum_{k=0}^{\infty}p^{\alpha_1}_{k}(1,t)=p^{\alpha_1}(1,0)-\lambda\sum_{k=0}^{\infty} I_t^{\alpha_1} \left(p^{\alpha_1}_{k}(1,t)-p^{\alpha_0}_{k}(0,t)\right).
	\end{equation*}
	Thus, $p^{\alpha_1}_{0}(1,t)=p^{\alpha_1}(1,0)=0$ and $p^{\alpha_1}_{k}(1,t)=-\lambda I_t^{\alpha_1} \left(p^{\alpha_1}_{k-1}(1,t)-p^{\alpha_0}_{k-1}(0,t)\right)$, $k\geq 1$.\\
	Hence,
	\begin{align*}
	p^{\alpha_1}_{1}(1,t)&=-\lambda I_t^{\alpha_1} \left(p^{\alpha_1}_{0}(1,t)-p^{\alpha_0}_{0}(0,t)\right)=\lambda I_t^{\alpha_1} t^0=-(-\lambda)\frac{ t^{\alpha_1}}{\Gamma({\alpha_1}+1)},\\
	p^{\alpha_1}_{2}(1,t)&=-\lambda I_t^{\alpha_1} \left(p^{\alpha_1}_{1}(1,t)-p^{\alpha_0}_{1}(0,t)\right)\\
	&=-(-\lambda)^2 \left(\frac{ t^{2\alpha_1}}{\Gamma(2\alpha_1+1)}+\frac{ t^{\alpha_0+\alpha_1}}{\Gamma(\alpha_0+\alpha_1+1)}\right),\\
	p^{\alpha_1}_{3}(1,t)&=-\lambda I_t^{\alpha_1} \left(p^{\alpha_1}_{2}(1,t)-p^{\alpha_0}_{2}(0,t)\right)\\
	&=-(-\lambda)^3 \left(\frac{ t^{3\alpha_1}}{\Gamma(3\alpha_1+1)}+\frac{ t^{\alpha_0+2\alpha_1}}{\Gamma(\alpha_0+2\alpha_1+1)}+\frac{ t^{2\alpha_0+\alpha_1}}{\Gamma(2\alpha_0+\alpha_1+1)}\right).
	\end{align*}
	Let
	\begin{equation}
	p^{\alpha_1}_{k-1}(1,t)=-(-\lambda)^{k-1}\underset{\Theta^{k-1}_{1}}{\sum}\frac{t^{k_0\alpha_0+k_1\alpha_1}}{\Gamma\left(k_0\alpha_0+k_1\alpha_1+1\right)}.
	\end{equation}
	Then
	\begin{align*}
	p^{\alpha_1}_{k}(1,t)&=-\lambda I_t^{\alpha_1} \left(p^{\alpha_1}_{k-1}(1,t)-p^{\alpha_0}_{k-1}(0,t)\right)\\
	&=-(-\lambda)^k\left(\underset{\Theta^{k-1}_{1}}{\sum}\frac{t^{k_0\alpha_0+k_1\alpha_1+\alpha_1}}{\Gamma\left(k_0\alpha_0+k_1\alpha_1+\alpha_1+1\right)}+\frac{t^{(k-1)\alpha_0+\alpha_1}}{\Gamma((k-1){\alpha_0}+\alpha_1+1)}\right)\\
	&=-(-\lambda)^k\underset{\Theta^{k}_{1}}{\sum}\frac{t^{k_0\alpha_0+k_1\alpha_1}}{\Gamma\left(k_0\alpha_0+k_1\alpha_1+1\right)},\ \ k\geq 1.
	\end{align*}
	Therefore,
	\begin{equation}
	p^{\alpha_1}(1,t)=-\sum_{k=1}^{\infty}(-\lambda)^k\underset{\Theta^{k}_{1}}{\sum}\frac{t^{k_0\alpha_0+k_1\alpha_1}}{\Gamma\left(k_0\alpha_0+k_1\alpha_1+1\right)},
	\end{equation}
	and thus the result holds for $n=1$.
	
	For $n=2$, substituting $p^{\alpha_2}(2,t)=\sum_{k=0}^{\infty}p^{\alpha_2}_{k}(2,t)$ in (\ref{yth}) and applying ADM, we get
	\begin{equation*}
	\sum_{k=0}^{\infty}p^{\alpha_2}_{k}(2,t)=p^{\alpha_2}(2,0)-\lambda\sum_{k=0}^{\infty} I_t^{\alpha_2} \left(p^{\alpha_2}_{k}(2,t)-p^{\alpha_1}_{k}(1,t)\right).
	\end{equation*}
	Thus, $p^{\alpha_2}_{0}(2,t)=p^{\alpha_2}(2,0)=0$ and $p^{\alpha_2}_{k}(2,t)=-\lambda I_t^{\alpha_2} \left(p^{\alpha_2}_{k-1}(2,t)-p^{\alpha_1}_{k-1}(1,t)\right)$, $k\geq 1$.\\
	Hence,
	\begin{align*}
	p^{\alpha_2}_{1}(2,t)&=-\lambda I_t^{\alpha_2} \left(p^{\alpha_2}_{0}(2,t)-p^{\alpha_1}_{0}(1,t)\right)=0,\\
	p^{\alpha_2}_{2}(2,t)&=-\lambda I_t^{\alpha_2} \left(p^{\alpha_2}_{1}(2,t)-p^{\alpha_1}_{1}(1,t)\right)=-\lambda I_t^{\alpha_2}\left(-\frac{\lambda t^{\alpha_1}}{\Gamma({\alpha_1}+1)}\right)=(-\lambda)^2\frac{ t^{\alpha_1+\alpha_2}}{\Gamma(\alpha_1+\alpha_2+1)},\\
	p^{\alpha_2}_{3}(2,t)&=-\lambda I_t^{\alpha_2} \left(p^{\alpha_2}_{2}(2,t)-p^{\alpha_1}_{2}(1,t)\right)\\
	&=(-\lambda)^3\left(\frac{ t^{\alpha_1+2\alpha_2}}{\Gamma(\alpha_1+2\alpha_2+1)}+\frac{t^{2\alpha_1+\alpha_2}}{\Gamma(2\alpha_1+\alpha_2+1)}+\frac{t^{\alpha_0+\alpha_1+\alpha_2}}{\Gamma(\alpha_0+\alpha_1+\alpha_2+1)}\right).
	\end{align*}
	Let
	\begin{equation}
	p^{\alpha_2}_{k-1}(2,t)=(-\lambda)^{k-1}\underset{\Theta^{k-1}_{2}}{\sum}\frac{t^{k_0\alpha_0+k_1\alpha_1+k_2\alpha_2}}{\Gamma\left(k_0\alpha_0+k_1\alpha_1+k_2\alpha_2+1\right)}.
	\end{equation}
	Then
	\begin{align*}
	p^{\alpha_2}_{k}(2,t)&=-\lambda I_t^{\alpha_2} \left(p^{\alpha_2}_{k-1}(2,t)-p^{\alpha_1}_{k-1}(1,t)\right)\\
    &=(-\lambda)^k \left(\underset{\Theta^{k-1}_{2}}{\sum}\frac{(-\lambda)^{k-1}t^{k_0\alpha_0+k_1\alpha_1+k_2\alpha_2+\alpha_2}}{\Gamma\left(k_0\alpha_0+k_1\alpha_1+k_2\alpha_2+\alpha_2+1\right)}+\underset{\Theta^{k-1}_{1}}{\sum}\frac{(-\lambda)^{k-1}t^{k_0\alpha_0+k_1\alpha_1+\alpha_2}}{\Gamma\left(k_0\alpha_0+k_1\alpha_1+\alpha_2+1\right)}\right)\\
	&=(-\lambda)^k\underset{\Theta^{k}_{2}}{\sum}\frac{t^{k_0\alpha_0+k_1\alpha_1+k_2\alpha_2}}{\Gamma\left(k_0\alpha_0+k_1\alpha_1+k_2\alpha_2+1\right)},\ \ k\geq 2.
	\end{align*}
	Therefore,
	\begin{equation}
	p^{\alpha_2}(2,t)=\sum_{k=2}^{\infty}(-\lambda)^k\underset{\Theta^{k}_{2}}{\sum}\frac{t^{k_0\alpha_0+k_1\alpha_1+k_2\alpha_2}}{\Gamma\left(k_0\alpha_0+k_1\alpha_1+k_2\alpha_2+1\right)},
	\end{equation}
	and thus the result holds for $n=2$.
	
	For $n=3$, substituting $p^{\alpha_3}(3,t)=\sum_{k=0}^{\infty}p^{\alpha_3}_{k}(3,t)$ in (\ref{yth}) and applying ADM, we get
	\begin{equation*}
	\sum_{k=0}^{\infty}p^{\alpha_3}_{k}(3,t)=p^{\alpha_3}(3,0)-\lambda\sum_{k=0}^{\infty} I_t^{\alpha_3} \left(p^{\alpha_3}_{k}(3,t)-p^{\alpha_2}_{k}(2,t)\right).
	\end{equation*}
	Thus, $p^{\alpha_3}_{0}(3,t)=p^{\alpha_3}(3,0)=0$ and $p^{\alpha_3}_{k}(3,t)=-\lambda I_t^{\alpha_3} \left(p^{\alpha_3}_{k-1}(3,t)-p^{\alpha_2}_{k-1}(2,t)\right)$, $k\geq 1$.\\
	Hence,
	\begin{align*}
	p^{\alpha_3}_{1}(3,t)&=-\lambda I_t^{\alpha_3} \left(p^{\alpha_3}_{0}(3,t)-p^{\alpha_2}_{0}(2,t)\right)=0,\\
	p^{\alpha_3}_{2}(3,t)&=-\lambda I_t^{\alpha_3} \left(p^{\alpha_3}_{1}(3,t)-p^{\alpha_2}_{1}(2,t)\right)=0,\\
	p^{\alpha_3}_{3}(3,t)&=-\lambda I_t^{\alpha_3} \left(p^{\alpha_3}_{2}(3,t)-p^{\alpha_2}_{2}(2,t)\right)\\
	&=-\lambda I_t^{\alpha_3} \left(-\frac{\lambda^2 t^{\alpha_1+\alpha_2}}{\Gamma(\alpha_1+\alpha_2+1)}\right)=-(-\lambda)^3\frac{ t^{\alpha_1+\alpha_2+\alpha_3}}{\Gamma(\alpha_1+\alpha_2+\alpha_3+1)},\\
	p^{\alpha_3}_{4}(3,t)&=-\lambda I_t^{\alpha_3} \left(p^{\alpha_3}_{3}(3,t)-p^{\alpha_2}_{3}(2,t)\right)\\
	&=-(-\lambda)^4\left(\frac{ t^{\alpha_1+\alpha_2+2\alpha_3}}{\Gamma(\alpha_1+\alpha_2+2\alpha_3+1)}+\frac{ t^{\alpha_1+2\alpha_2+\alpha_3}}{\Gamma(\alpha_1+2\alpha_2+\alpha_3+1)}\right.\\
	&\ \ \ \ \ \ \ \ \ \ \ \ \ \ \ \ \ \left.+\frac{t^{2\alpha_1+\alpha_2+\alpha_3}}{\Gamma(2\alpha_1+\alpha_2+\alpha_3+1)}+\frac{t^{\alpha_0+\alpha_1+\alpha_2+\alpha_3}}{\Gamma(\alpha_0+\alpha_1+\alpha_2+\alpha_3+1)}\right).
	\end{align*}
	Let
	\begin{equation}
	p^{\alpha_3}_{k-1}(3,t)=-(-\lambda)^{k-1}\underset{\Theta^{k-1}_{3}}{\sum}\frac{t^{k_0\alpha_0+k_1\alpha_1+k_2\alpha_2+k_3\alpha_3}}{\Gamma\left(k_0\alpha_0+k_1\alpha_1+k_2\alpha_2+k_3\alpha_3+1\right)}.
	\end{equation}
	Then
	\begin{align*}
	p^{\alpha_3}_{k}(3,t)&=-\lambda I_t^{\alpha_3} \left(p^{\alpha_3}_{k-1}(3,t)-p^{\alpha_2}_{k-1}(2,t)\right)\\
    &=-(-\lambda)^k\left(\underset{\Theta^{k-1}_{3}}{\sum}\frac{t^{\sum_{j=0}^3k_j\alpha_j+\alpha_3}}{\Gamma\left(\sum_{j=0}^3k_j\alpha_j+\alpha_3+1\right)}+\underset{\Theta^{k-1}_{2}}{\sum}\frac{t^{\sum_{j=0}^2k_j\alpha_j+\alpha_3}}{\Gamma\left(\sum_{j=0}^2k_j\alpha_j+\alpha_3+1\right)}\right)\\
	&=-(-\lambda)^k\underset{\Theta^{k}_{3}}{\sum}\frac{t^{\sum_{j=0}^3k_j\alpha_j}}{\Gamma\left(\sum_{j=0}^3k_j\alpha_j+1\right)},\ \ k\geq 3.
	\end{align*}
	Therefore,
	\begin{equation}
	p^{\alpha_3}(3,t)=-\sum_{k=3}^{\infty}(-\lambda)^k\underset{\Theta^{k}_{3}}{\sum}\frac{t^{\sum_{j=0}^3k_j\alpha_j}}{\Gamma\left(\sum_{j=0}^3k_j\alpha_j+1\right)},
	\end{equation}
	and thus the result holds for $n=3$.
	
	Let $p^{\alpha_m}(m,t)=\sum_{k=0}^{\infty}p^{\alpha_m}_{k}(m,t)$ in (\ref{yth}) and assume the result holds for $n=m>3$, {\it i.e.}, $p^{\alpha_m}_{k}(m,t)=0$, $k<m$ and
	\begin{equation*}
	p^{\alpha_m}_{k}(m,t)=(-1)^m(-\lambda)^k\underset{\Theta^{k}_{m}}{\sum}\frac{t^{\sum_{j=0}^mk_j\alpha_j}}{\Gamma\left(\sum_{j=0}^mk_j\alpha_j+1\right)},\ \ k\geq m.
	\end{equation*}
	For $n=m+1$, substituting $p^{\alpha_{m+1}}(m+1,t)=\sum_{k=0}^{\infty}p^{\alpha_{m+1}}_{k}(m+1,t)$ in (\ref{yth}) and applying ADM, we get
	\begin{equation*}
	\sum_{k=0}^{\infty}p^{\alpha_{m+1}}_{k}(m+1,t)=p^{\alpha_{m+1}}(m+1,0)-\lambda\sum_{k=0}^{\infty} I_t^{\alpha_{m+1}} \left(p^{\alpha_{m+1}}_{k}(m+1,t)-p^{\alpha_{m}}_{k}(m,t)\right).
	\end{equation*}
	Thus, $p^{\alpha_{m+1}}_{0}(m+1,t)=p^{\alpha_{m+1}}(m+1,0)=0$ and\\ $p^{\alpha_{m+1}}_{k}(m+1,t)=-\lambda I_t^{\alpha_{m+1}} \left(p^{\alpha_{m+1}}_{k-1}(m+1,t)-p^{\alpha_{m}}_{k-1}(m,t)\right)$, $k\geq 1$. Hence,
	\begin{align*}
	p^{\alpha_{m+1}}_{1}(m+1,t)&=-\lambda I_t^{\alpha_{m+1}} \left(p^{\alpha_{m+1}}_{0}(m+1,t)-p^{\alpha_{m}}_{0}(m,t)\right)=0,\\
	p^{\alpha_{m+1}}_{2}(m+1,t)&=-\lambda I_t^{\alpha_{m+1}} \left(p^{\alpha_{m+1}}_{1}(m+1,t)-p^{\alpha_{m}}_{1}(m,t)\right)=0.
	\end{align*}
	Let
	\begin{equation*}
	p^{\alpha_{m+1}}_{k-1}(m+1,t)=0,\ \ k-1<m+1.
	\end{equation*}
	Then
	\begin{equation*}
	p^{\alpha_{m+1}}_{k}(m+1,t)=-\lambda I_t^{\alpha_{m+1}} \left(p^{\alpha_{m+1}}_{k-1}(m+1,t)-p^{\alpha_{m}}_{k-1}(m,t)\right)=0,\ \ k<m+1.
	\end{equation*}
	Now for $k\geq m+1$, we have
	\begin{align*}
	p^{\alpha_{m+1}}_{m+1}(m+1,t)&=-\lambda I_t^{\alpha_{m+1}} \left(p^{\alpha_{m+1}}_{m}(m+1,t)-p^{\alpha_{m}}_{m}(m,t)\right)\\
	&=-\lambda I_t^{\alpha_{m+1}} \left(-\underset{\Theta^{m}_{m}}{\sum}\frac{\lambda^mt^{\sum_{j=0}^mk_j\alpha_j}}{\Gamma\left(\sum_{j=0}^mk_j\alpha_j+1\right)}\right)\\
	&=\lambda^{m+1} \underset{\Theta^{m}_{m}}{\sum}\frac{t^{\sum_{j=0}^mk_j\alpha_j+\alpha_{m+1}}}{\Gamma\left(\sum_{j=0}^mk_j\alpha_j+\alpha_{m+1}+1\right)}\\
	&=(-1)^{m+1}(-\lambda)^{m+1} \underset{\Theta^{m+1}_{m+1}}{\sum}\frac{t^{\sum_{j=0}^{m+1}k_j\alpha_j}}{\Gamma\left(\sum_{j=0}^{m+1}k_j\alpha_j+1\right)},\\
	p^{\alpha_{m+1}}_{m+2}(m+1,t)&=-\lambda I_t^{\alpha_{m+1}} \left(p^{\alpha_{m+1}}_{m+1}(m+1,t)-p^{\alpha_{m}}_{m+1}(m,t)\right)\\
	&=-\lambda^{m+2}\left(\underset{\Theta^{m+1}_{m+1}}{\sum}\frac{ t^{\sum_{j=0}^{m+1}k_j\alpha_j+\alpha_{m+1}}}{\Gamma\left(\sum_{j=0}^{m+1}k_j\alpha_j+\alpha_{m+1}+1\right)}\right.\\
	&\ \ \ \ \ \ \ \ \ \ \ \ \ \ \ \ \ \ \ \ \ \ \ \left.+\underset{\Theta^{m+1}_{m}}{\sum}\frac{t^{\sum_{j=0}^mk_j\alpha_j+\alpha_{m+1}}}{\Gamma\left(\sum_{j=0}^mk_j\alpha_j+\alpha_{m+1}+1\right)}\right)\\
	&=(-1)^{m+1}(-\lambda)^{m+2}\underset{\Theta^{m+2}_{m+1}}{\sum}\frac{t^{\sum_{j=0}^{m+1}k_j\alpha_j}}{\Gamma\left(\sum_{j=0}^{m+1}k_j\alpha_j+1\right)}.
	\end{align*}
	Let
	\begin{equation*}
	p^{\alpha_{m+1}}_{k-1}(m+1,t)=(-1)^{m+1}(-\lambda)^{k-1}\underset{\Theta^{k-1}_{m+1}}{\sum}\frac{t^{\sum_{j=0}^{m+1}k_j\alpha_j}}{\Gamma\left(\sum_{j=0}^{m+1}k_j\alpha_j+1\right)},\ \ k-1\geq m+1.
	\end{equation*}
	Then
	\begin{align*}
	p^{\alpha_{m+1}}_{k}(m+1,t)&=-\lambda I_t^{\alpha_{m+1}} \left(p^{\alpha_{m+1}}_{k-1}(m+1,t)-p^{\alpha_{m}}_{k-1}(m,t)\right)\\
    &=(-1)^{m+1}(-\lambda)^k \left(\underset{\Theta^{k-1}_{m+1}}{\sum}\frac{t^{\sum_{j=0}^{m+1}k_j\alpha_j+\alpha_{m+1}}}{\Gamma\left(\sum_{j=0}^{m+1}k_j\alpha_j+\alpha_{m+1}+1\right)}\right.\\
    &\ \ \ \ \ \ \ \ \ \ \ \ \ \ \ \ \ \ \ \ \ \ \ \ \ \ \ \ \ \ \ \ \ \left.+\underset{\Theta^{k-1}_{m}}{\sum}\frac{t^{\sum_{j=0}^mk_j\alpha_j+\alpha_{m+1}}}{\Gamma\left(\sum_{j=0}^mk_j\alpha_j+\alpha_{m+1}+1\right)}\right)\\
	&=(-1)^{m+1}(-\lambda)^k\underset{\Theta^{k}_{m+1}}{\sum}\frac{t^{\sum_{j=0}^{m+1}k_j\alpha_j}}{\Gamma\left(\sum_{j=0}^{m+1}k_j\alpha_j+1\right)},\ \ k\geq m+1.
	\end{align*}
	Therefore,
	\begin{align*}
	p^{\alpha_{m+1}}(m+1,t)&=(-1)^{m+1}\sum_{k=m+1}^{\infty}(-\lambda)^k\underset{\Theta^{k}_{m+1}}{\sum}\frac{t^{\sum_{j=0}^{m+1}k_j\alpha_j}}{\Gamma\left(\sum_{j=0}^{m+1}k_j\alpha_j+1\right)},
	\end{align*}
	and thus the result holds for $n=m+1$. This completes the proof.
\end{proof}
\begin{remark}
Note that (\ref{2.4kky}) can also be expressed in the following form
\begin{equation*}\label{tea}
p^{\alpha_n}(n,t)=\lambda^nt^{\sum_{j=1}^n\alpha_j}\sum_{k=0}^{\infty‎}(-\lambda)^k\underset{k_j\in\mathbb{N}_0}{\underset{\sum_{j=0}^nk_j=k}{\sum}}\frac{t^{\sum_{j=0}^nk_j\alpha_j}}{\Gamma(1+k_0\alpha_0+\sum_{j=1}^n(k_j+1)\alpha_j)},\ \ n\geq0.
\end{equation*}
\end{remark}
The Laplace transform of the state probabilities of SDTFPP-I can also be obtained from the above result as follows:
   \begin{align*}
	\tilde{p}^{\alpha_n}(n,s)&=\int_0^{\infty}p^{\alpha_n}(n,t)e^{-st}\,\mathrm{d}t\\
	&=(-1)^n\sum_{k=n}^{\infty‎}\underset{\Theta^k_n}{\sum}\frac{(-\lambda)^k}{s^{k_0\alpha_0+k_1\alpha_1+\cdots+k_n\alpha_n+1}},\ \ \ \ \ \ \ \mathrm{(using}\ \mathrm{(}\ref{2.4kky}\mathrm{))}\\
	&=\frac{\lambda^n}{s^{\alpha_1+\alpha_2+\cdots+\alpha_n+1}}\sum_{k=n}^{\infty‎}\underset{k_0\in\mathbb{N}_0,\ k_j\in\mathbb{N}_0\backslash\{0\},\ 1\leq j\leq n}{\underset{k_0+k_1+\cdots+k_n=k}{\sum}}\frac{(-\lambda)^{k_0+(k_1-1)+\cdots+(k_n-1)}}{s^{k_0\alpha_0+(k_1-1)\alpha_1+\cdots+(k_n-1)\alpha_n}}\\
	&=\frac{\lambda^ns^{\alpha_0-1}}{s^{\alpha_0+\alpha_1+\cdots+\alpha_n}}\sum_{k=0}^{\infty‎}\underset{l_j\in\mathbb{N}_0,\ 0\leq j\leq n}{\underset{l_0+l_1+\cdots+l_n=k}{\sum}}\frac{(-\lambda)^{l_0+l_1+\cdots+l_n}}{s^{l_0\alpha_0+l_1\alpha_1+\cdots+l_n\alpha_n}},\ \ (l_0=k_0,\ l_j=k_j-1,\ j\neq0)\\
	&=\frac{\lambda^ns^{\alpha_0-1}}{s^{\alpha_0+\alpha_1+\cdots+\alpha_n}}\sum_{l_0=0}^{\infty‎}\sum_{l_1=0}^{l_0}\sum_{l_2=0}^{l_1}\cdots\sum_{l_n=0}^{l_{n-1}}\left(\frac{-\lambda}{s^{\alpha_0}}\right)^{l_n}\left(\frac{-\lambda}{s^{\alpha_1}}\right)^{l_{n-1}-l_n}\cdots\left(\frac{-\lambda}{s^{\alpha_n}}\right)^{l_0-l_1}\\
	&=\frac{\lambda^ns^{\alpha_0-1}}{s^{\alpha_0+\alpha_1+\cdots+\alpha_n}}\prod_{k=0}^{n}\sum_{l_k=0}^{\infty‎}\left(\frac{-\lambda}{s^{\alpha_k}}\right)^{l_k}\\
	&=\frac{\lambda^ns^{\alpha_0-1}}{s^{\alpha_0+\alpha_1+\cdots+\alpha_n}}\prod_{k=0}^{n}\left(1+\frac{\lambda}{s^{\alpha_k}}\right)^{-1}\\
	&=\frac{\lambda^ns^{\alpha_0-1}}{\prod_{k=0}^n(s^{\alpha_k}+\lambda)},
	\end{align*}
	which coincides with Theorem 2.1 of Garra {\it et al.} (2015).
	 
	An application of Theorem \ref{t1} is the following result.
	\begin{corollary}\label{cdb}
		Let $X_0,X_1,\ldots,X_n$ be $n+1$ independent random variables such that $X_j$ follows the Mittag-Leffler distribution (see Pillai (1990)) with distribution function $F_{X_j}(t)=1-E_{\alpha_j}(-t^{\alpha_j})$, where $0<\alpha_j\leq1$, $0\leq j\leq n$ and $\alpha_0=1$. Then, the density function of the convolution $T=X_0+X_1+\cdots+X_n$ is given by
		\begin{equation*}
		f_T(t)=\sum_{k=n}^{\infty‎}(-1)^{n+k}\underset{\Theta^k_n}{\sum}\frac{t^{k_0+\sum_{j=1}^nk_j\alpha_j}}{\Gamma(1+k_0+\sum_{j=1}^nk_j\alpha_j)},\ \ t\geq0,
		\end{equation*}
		where 
		$
		\Theta^k_n=\{(k_0,k_1,\ldots,k_n):\ \sum_{j=0}^nk_j=k,\ k_0\in\mathbb{N}_0,\ k_j\in\mathbb{N}_0\backslash\{0\},\ 1\leq j\leq n\}.
		$
	\end{corollary}
	\begin{proof}
		The Laplace transform of the density of Mittag-Leffler random variables $X_j$'s is
		\begin{equation*}
		\tilde{f}_{X_j}(s)=\mathbb{E}(e^{-sX_j})=\frac{1}{1+s^{\alpha_j}},\ \ s>0.
		\end{equation*}
		Hence, the Laplace transform of the convolution is given by
		\begin{equation*}
		\tilde{f}_{T}(s)=\mathbb{E}(e^{-sT})=\prod_{j=0}^{n}\mathbb{E}(e^{-sX_j})=\frac{1}{\prod_{j=0}^n(1+s^{\alpha_j})},\ \ s>0.
		\end{equation*}
		The proof follows in view of (\ref{lt1}) and (\ref{2.4kky}). 
	\end{proof}
\begin{corollary}
		Let $X_1$ be the first waiting time of SDTFPP-I. Then the distribution of $X_1$ is given by:
	\begin{equation}
	\mathrm{Pr}\{X_1>t\}=\mathrm{Pr}\{N_1(t,\lambda)=0\}=E_{\alpha_0}(-\lambda t^{\alpha_0}),
	\end{equation}
	where $E_{\alpha_0}(.)$ is the Mittag-Leffler function given by (\ref{fastt}).
\end{corollary}
The distribution of TFPP can be obtained as a special case of SDTFPP-I.
\begin{corollary}
	Let $\alpha_n=\alpha$ for all $n\geq0$. Then
	\begin{equation}
	p^\alpha(n,t)=\frac{(\lambda t^\alpha)^n}{n!}\sum_{k=0}^{\infty‎}\frac{(k+n)!}{k!}\frac{(-\lambda t^\alpha)^k}{\Gamma\left((k+n)\alpha+1\right)},\ \ 0<\alpha\leq1,\ \lambda>0,\ n\geq0,
	\end{equation}
	 which is the distribution of TFPP.
\end{corollary}
\begin{proof}
	On substituting $\alpha_n=\alpha$ for all $n\geq0$ in (\ref{2.4kky}), we obtain
	\begin{equation*}
	p^{\alpha}(n,t)=(-1)^n\sum_{k=n}^{\infty‎}(-\lambda)^k|\Theta^k_n|\frac{t^{k\alpha}}{\Gamma\left(k\alpha+1\right)},
	\end{equation*}
	where $|\Theta^k_n|$ denotes the cardinality of the set $\Theta^k_n$. 
	
	The number of positive integral solutions of $k_0+k_1+\cdots+k_n=k$ is $\binom{k-1}{n}$ (see Proposition 6.1., Ross (2010)), and similarly the number of positive integral solutions of $k_1+k_2+\cdots+k_n=k$ is $\binom{k-1}{n-1}$. Hence, the cardinality of the set $\Theta^k_n$ is $|\Theta^k_n|=\binom{k-1}{n}+\binom{k-1}{n-1}=\binom{k}{n}$. Therefore,
	\begin{equation*}
	p^{\alpha}(n,t)=(-1)^n\sum_{k=n}^{\infty‎}\frac{k!}{n!(k-n)!}\frac{(-\lambda t^{\alpha})^k}{\Gamma\left(k\alpha+1\right)}=\frac{(\lambda t^{\alpha})^n}{n!}\sum_{k=0}^{\infty‎}\frac{(k+n)!}{k!}\frac{(-\lambda t^{\alpha})^k}{\Gamma\left((k+n)\alpha+1\right)},\ \ n\geq0,
	\end{equation*}
	and thus the proof follows.
\end{proof}

\section{State dependent time fractional Poisson process-II}
In this section, we obtain the state probabilities of SDTFPP-II, a point process constructed by Garra {\it et al.} (2015) by considering independent and nonidentically distributed waiting times.
\begin{theorem}\label{t2}
	Consider the following difference-differential equations governing the state probabilities of SDTFPP-II:
	\begin{equation}\label{wqr2new}
	p^{\beta_n}(n,t)=p^{\beta_n}(n,0)-\lambda(I_t^{\beta_n}p^{\beta_n}(n,t)-I_t^{\beta_{n-1}}p^{\beta_{n-1}}(n-1,t)),\ \ 0<\beta_n\leq 1,\ \lambda>0,\ n\geq 0,
	\end{equation}
	with $p^{\beta_0}(0,0)=1$ and $p^{\beta_n}(n,0)=0$, $n\geq 1$. The solution of (\ref{wqr2new}) is given by
	\begin{equation}\label{wq2.4kky}
	p^{\beta_n}(n,t)=(-1)^n\sum_{k=n}^{\infty‎}(-\lambda)^k\underset{\Omega^k_n}{\sum}\frac{t^{\sum_{j=0}^nk_j\beta_j}}{\Gamma(1+\sum_{j=0}^nk_j\beta_j)},\ \ n\geq0,
	\end{equation}
	where 
	$
	\Omega^k_n=\{(k_0,k_1,\ldots,k_n):\ \sum_{j=0}^nk_j=k,\ k_n\in\mathbb{N}_0,\ k_j\in\mathbb{N}_0\backslash\{0\},\ 0\leq j\leq n-1\}.
	$
\end{theorem}
\begin{proof}
	Note that $p^{\beta_{-1}}(-1,t)=0$ for $t\geq0$.
	
	 The case $n=0$ for SDTFPP-II corresponds to the case $n=0$ of SDTFPP-I. Hence, on substituting $p^{\beta_0}(0,t)=\sum_{k=0}^{\infty}p^{\beta_0}_{k}(0,t)$ in (\ref{wqr2new}), we obtain $p^{\beta_0}_{0}(0,t)=p^{\beta_0}(0,0)=1$ and
	\begin{equation}
	p^{\beta_0}_{k}(0,t)=\frac{(-\lambda t^{\beta_0})^{k}}{\Gamma(k\beta_0+1)},\ \ k\geq 0.
	\end{equation}
	
	For $n=1$, substituting $p^{\beta_1}(1,t)=\sum_{k=0}^{\infty}p^{\beta_1}_{k}(1,t)$ in (\ref{wqr2new}) and applying ADM, we get
	\begin{equation*}
	\sum_{k=0}^{\infty}p^{\beta_1}_{k}(1,t)=p^{\beta_1}(1,0)-\lambda\sum_{k=0}^{\infty} \left(I_t^{\beta_1}p^{\beta_1}_{k}(1,t)-I_t^{\beta_0}p^{\beta_0}_{k}(0,t)\right).
	\end{equation*}
	Thus, $p^{\beta_1}_{0}(1,t)=p^{\beta_1}(1,0)=0$ and $p^{\beta_1}_{k}(1,t)=-\lambda \left(I_t^{\beta_1}p^{\beta_1}_{k-1}(1,t)-I_t^{\beta_0}p^{\beta_0}_{k-1}(0,t)\right)$, $k\geq 1$.\\
	Hence,
	\begin{align*}
	p^{\beta_1}_{1}(1,t)&=-\lambda\left(I_t^{\beta_1}p^{\beta_1}_{0}(1,t)-I_t^{\beta_0}p^{\beta_0}_{0}(0,t)\right)=\lambda I_t^{\beta_0} t^0=-(-\lambda)\frac{ t^{\beta_0}}{\Gamma({\beta_0}+1)},\\
	p^{\beta_1}_{2}(1,t)&=-\lambda\left(I_t^{\beta_1}p^{\beta_1}_{1}(1,t)-I_t^{\beta_0}p^{\beta_0}_{1}(0,t)\right)\\
	&=-\lambda\left(\frac{\lambda I_t^{\beta_1}t^{\beta_0}}{\Gamma(\beta_0+1)}+\frac{\lambda I_t^{\beta_0}t^{\beta_0}}{\Gamma(\beta_0+1)}\right)=-(-\lambda)^2 \left(\frac{ t^{\beta_0+\beta_1}}{\Gamma(\beta_0+\beta_1+1)}+\frac{ t^{2\beta_0}}{\Gamma(2\beta_0+1)}\right),\\
	p^{\beta_1}_{3}(1,t)&=-\lambda\left(I_t^{\beta_1}p^{\beta_1}_{2}(1,t)-I_t^{\beta_0}p^{\beta_0}_{2}(0,t)\right)\\
  	&=-(-\lambda)^3 \left(\frac{ t^{\beta_0+2\beta_1}}{\Gamma(\beta_0+2\beta_1+1)}+\frac{ t^{2\beta_0+\beta_1}}{\Gamma(2\beta_0+\beta_1+1)}+\frac{ t^{3\beta_0}}{\Gamma(3\beta_0+1)}\right).
	\end{align*}
	Let
	\begin{equation}
	p^{\beta_1}_{k-1}(1,t)=-(-\lambda)^{k-1}\underset{\Omega^{k-1}_{1}}{\sum}\frac{t^{k_0\beta_0+k_1\beta_1}}{\Gamma\left(k_0\beta_0+k_1\beta_1+1\right)}.
	\end{equation}
	Then
	\begin{align*}
	p^{\beta_1}_{k}(1,t)&=-\lambda\left(I_t^{\beta_1}p^{\beta_1}_{k-1}(1,t)-I_t^{\beta_0}p^{\beta_0}_{k-1}(0,t)\right)\\
	&=-\lambda\left(-\underset{\Omega^{k-1}_{1}}{\sum}\frac{(-\lambda)^{k-1}t^{k_0\beta_0+k_1\beta_1+\beta_1}}{\Gamma\left(k_0\beta_0+k_1\beta_1+\beta_1+1\right)}-\frac{(-\lambda )^{k-1}t^{k\beta_0}}{\Gamma(k\beta_0+1)}\right)\\&=-(-\lambda)^k\underset{\Omega^{k}_{1}}{\sum}\frac{t^{k_0\beta_0+k_1\beta_1}}{\Gamma\left(k_0\beta_0+k_1\beta_1+1\right)},\ \ k\geq 1.
	\end{align*}
	Therefore,
	\begin{equation}
	p^{\beta_1}(1,t)=-\sum_{k=1}^{\infty}(-\lambda)^k\underset{\Omega^{k}_{1}}{\sum}\frac{t^{k_0\beta_0+k_1\beta_1}}{\Gamma\left(k_0\beta_0+k_1\beta_1+1\right)},
	\end{equation}
	and thus the result holds for $n=1$.
	
	For $n=2$, substituting $p^{\beta_2}(2,t)=\sum_{k=0}^{\infty}p^{\beta_2}_{k}(2,t)$ in (\ref{wqr2new}) and applying ADM, we get
	\begin{equation*}
	\sum_{k=0}^{\infty}p^{\beta_2}_{k}(2,t)=p^{\beta_2}(2,0)-\lambda\sum_{k=0}^{\infty}\left(I_t^{\beta_2}p^{\beta_2}_{k}(2,t)-I_t^{\beta_1}p^{\beta_1}_{k}(1,t)\right).
	\end{equation*}
	Thus, $p^{\beta_2}_{0}(2,t)=p^{\beta_2}(2,0)=0$ and $p^{\beta_2}_{k}(2,t)=-\lambda\left(I_t^{\beta_2}p^{\beta_2}_{k-1}(2,t)-I_t^{\beta_1}p^{\beta_1}_{k-1}(1,t)\right)$, $k\geq 1$.\\
	Hence,
	\begin{align*}
	p^{\beta_2}_{1}(2,t)&=-\lambda\left(I_t^{\beta_2}p^{\beta_2}_{0}(2,t)-I_t^{\beta_1}p^{\beta_1}_{0}(1,t)\right)=0,\\
	p^{\beta_2}_{2}(2,t)&=-\lambda\left(I_t^{\beta_2}p^{\beta_2}_{1}(2,t)-I_t^{\beta_1}p^{\beta_1}_{1}(1,t)\right)=(-\lambda)^2\frac{I_t^{\beta_1}t^{\beta_0}}{\Gamma({\beta_0}+1)}=(-\lambda)^2\frac{ t^{\beta_0+\beta_1}}{\Gamma(\beta_0+\beta_1+1)},\\
	p^{\beta_2}_{3}(2,t)&=-\lambda\left(I_t^{\beta_2}p^{\beta_2}_{2}(2,t)-I_t^{\beta_1}p^{\beta_1}_{2}(1,t)\right)\\
	&=(-\lambda)^3\left(\frac{t^{\beta_0+\beta_1+\beta_2}}{\Gamma(\beta_0+\beta_1+\beta_2+1)}+\frac{t^{\beta_0+2\beta_1}}{\Gamma(\beta_0+2\beta_1+1)}+\frac{t^{2\beta_0+\beta_1}}{\Gamma(2\beta_0+\beta_1+1)}\right).
	\end{align*}
	Let
	\begin{equation}
	p^{\beta_2}_{k-1}(2,t)=(-\lambda)^{k-1}\underset{\Omega^{k-1}_{2}}{\sum}\frac{t^{k_0\beta_0+k_1\beta_1+k_2\beta_2}}{\Gamma\left(k_0\beta_0+k_1\beta_1+k_2\beta_2+1\right)}.
	\end{equation}
	Then
	\begin{align*}
	p^{\beta_2}_{k}(2,t)&=-\lambda\left(I_t^{\beta_2}p^{\beta_2}_{k-1}(2,t)-I_t^{\beta_1}p^{\beta_1}_{k-1}(1,t)\right)\\
	&=-\lambda\left(\underset{\Omega^{k-1}_{2}}{\sum}\frac{(-\lambda)^{k-1}t^{k_0\beta_0+k_1\beta_1+k_2\beta_2+\beta_2}}{\Gamma\left(k_0\beta_0+k_1\beta_1+k_2\beta_2+\beta_2+1\right)}+\underset{\Omega^{k-1}_{1}}{\sum}\frac{(-\lambda)^{k-1}t^{k_0\beta_0+k_1\beta_1+\beta_1}}{\Gamma\left(k_0\beta_0+k_1\beta_1+\beta_1+1\right)}\right)\\
	&=(-\lambda)^k\underset{\Omega^{k}_{2}}{\sum}\frac{t^{k_0\beta_0+k_1\beta_1+k_2\beta_2}}{\Gamma\left(k_0\beta_0+k_1\beta_1+k_2\beta_2+1\right)},\ \ k\geq 2.
	\end{align*}
	Therefore,
	\begin{equation}
	p^{\beta_2}(2,t)=\sum_{k=2}^{\infty}(-\lambda)^k\underset{\Omega^{k}_{2}}{\sum}\frac{t^{k_0\beta_0+k_1\beta_1+k_2\beta_2}}{\Gamma\left(k_0\beta_0+k_1\beta_1+k_2\beta_2+1\right)},
	\end{equation}
	and thus the result holds for $n=2$.
	
	For $n=3$, substituting $p^{\beta_3}(3,t)=\sum_{k=0}^{\infty}p^{\beta_3}_{k}(3,t)$ in (\ref{wqr2new}) and applying ADM, we get
	\begin{equation*}
	\sum_{k=0}^{\infty}p^{\beta_3}_{k}(3,t)=p^{\beta_3}(3,0)-\lambda\sum_{k=0}^{\infty}\left(I_t^{\beta_3}p^{\beta_3}_{k}(3,t)-I_t^{\beta_2}p^{\beta_2}_{k}(2,t)\right).
	\end{equation*}
	Thus, $p^{\beta_3}_{0}(3,t)=p^{\beta_3}(3,0)=0$ and $p^{\beta_3}_{k}(3,t)=-\lambda\left(I_t^{\beta_3}p^{\beta_3}_{k-1}(3,t)-I_t^{\beta_2}p^{\beta_2}_{k-1}(2,t)\right)$, $k\geq 1$.\\
	Hence,
	\begin{align*}
	p^{\beta_3}_{1}(3,t)&=-\lambda\left(I_t^{\beta_3}p^{\beta_3}_{0}(3,t)-I_t^{\beta_2}p^{\beta_2}_{0}(2,t)\right)=0,\\
	p^{\beta_3}_{2}(3,t)&=-\lambda\left(I_t^{\beta_3}p^{\beta_3}_{1}(3,t)-I_t^{\beta_2}p^{\beta_2}_{1}(2,t)\right)=0,\\
	p^{\beta_3}_{3}(3,t)&=-\lambda\left(I_t^{\beta_3}p^{\beta_3}_{2}(3,t)-I_t^{\beta_2}p^{\beta_2}_{2}(2,t)\right)\\
	&=-(-\lambda)^3\frac{I_t^{\beta_2}t^{\beta_0+\beta_1}}{\Gamma(\beta_0+\beta_1+1)}=-(-\lambda)^3\frac{ t^{\beta_0+\beta_1+\beta_2}}{\Gamma(\beta_0+\beta_1+\beta_2+1)},\\
	p^{\beta_3}_{4}(3,t)&=-\lambda\left(I_t^{\beta_3}p^{\beta_3}_{3}(3,t)-I_t^{\beta_2}p^{\beta_2}_{3}(2,t)\right)\\
	&=-(-\lambda)^4\left(\frac{t^{\beta_0+\beta_1+\beta_2+\beta_3}}{\Gamma(\beta_0+\beta_1+\beta_2+\beta_3+1)}+\frac{t^{\beta_0+\beta_1+2\beta_2}}{\Gamma(\beta_0+\beta_1+2\beta_2+1)}\right.\\
	&\ \ \ \ \ \ \ \ \ \ \ \ \ \ \ \ \ \ \ \ \ \ \ \ \ \ \left.+\frac{ t^{\beta_0+2\beta_1+\beta_2}}{\Gamma(\beta_0+2\beta_1+\beta_2+1)}+\frac{t^{2\beta_0+\beta_1+\beta_2}}{\Gamma(2\beta_0+\beta_1+\beta_2+1)}\right).
	\end{align*}
	Let
	\begin{equation}
	p^{\beta_3}_{k-1}(3,t)=-(-\lambda)^{k-1}\underset{\Omega^{k-1}_{3}}{\sum}\frac{t^{\sum_{j=0}^3k_j\beta_j}}{\Gamma\left(\sum_{j=0}^3k_j\beta_j+1\right)}.
	\end{equation}
	Then
	\begin{align*}
	p^{\beta_3}_{k}(3,t)&=-\lambda\left(I_t^{\beta_3}p^{\beta_3}_{k-1}(3,t)-I_t^{\beta_2}p^{\beta_2}_{k-1}(2,t)\right)\\
	&=-\lambda\left(-\underset{\Omega^{k-1}_{3}}{\sum}\frac{(-\lambda)^{k-1}t^{\sum_{j=0}^3k_j\beta_j+\beta_3}}{\Gamma\left(\sum_{j=0}^3k_j\beta_j+\beta_3+1\right)}-\underset{\Omega^{k-1}_{2}}{\sum}\frac{(-\lambda)^{k-1}t^{\sum_{j=0}^2k_j\beta_j+\beta_2}}{\Gamma\left(\sum_{j=0}^2k_j\beta_j+\beta_2+1\right)}\right)\\
	&=-(-\lambda)^k\underset{\Omega^{k}_{3}}{\sum}\frac{t^{\sum_{j=0}^3k_j\beta_j}}{\Gamma\left(\sum_{j=0}^3k_j\beta_j+1\right)},\ \ k\geq 3.
	\end{align*}
	Therefore,
	\begin{equation}
	p^{\beta_3}(3,t)=-\sum_{k=3}^{\infty}(-\lambda)^k\underset{\Omega^{k}_{3}}{\sum}\frac{t^{\sum_{j=0}^3k_j\beta_j}}{\Gamma\left(\sum_{j=0}^3k_j\beta_j+1\right)},
	\end{equation}
	and thus the result holds for $n=3$.
	
	Let $p^{\beta_m}(m,t)=\sum_{k=0}^{\infty}p^{\beta_m}_{k}(m,t)$ in (\ref{wqr2new}) and assume the result holds for $n=m>3$, {\it i.e.}, $p^{\beta_m}_{k}(m,t)=0$, $k<m$ and
	\begin{equation*}
	p^{\beta_m}_{k}(m,t)=(-1)^m(-\lambda)^k\underset{\Omega^{k}_{m}}{\sum}\frac{t^{\sum_{j=0}^mk_j\beta_j}}{\Gamma\left(\sum_{j=0}^mk_j\beta_j+1\right)},\ \ k\geq m.
	\end{equation*}
	
	Next consider the case $n=m+1$. Let now $p^{\beta_{m+1}}(m+1,t)=\sum_{k=0}^{\infty}p^{\beta_{m+1}}_{k}(m+1,t)$ in (\ref{wqr2new}) and applying ADM, we get
	\begin{equation*}
	\sum_{k=0}^{\infty}p^{\beta_{m+1}}_{k}(m+1,t)=p^{\beta_{m+1}}(m+1,0)-\lambda\sum_{k=0}^{\infty}  \left(I_t^{\beta_{m+1}}p^{\beta_{m+1}}_{k}(m+1,t)-I_t^{\beta_{m}}p^{\beta_{m}}_{k}(m,t)\right).
	\end{equation*}
	Thus, $p^{\beta_{m+1}}_{0}(m+1,t)=p^{\beta_{m+1}}(m+1,0)=0$ and\\ $p^{\beta_{m+1}}_{k}(m+1,t)=-\lambda\left(I_t^{\beta_{m+1}}p^{\beta_{m+1}}_{k-1}(m+1,t)-I_t^{\beta_{m}}p^{\beta_{m}}_{k-1}(m,t)\right)$, $k\geq 1$. Hence,
	\begin{align*}
	p^{\beta_{m+1}}_{1}(m+1,t)&=-\lambda\left(I_t^{\beta_{m+1}}p^{\beta_{m+1}}_{0}(m+1,t)-I_t^{\beta_{m}}p^{\beta_{m}}_{0}(m,t)\right)=0,\\
	p^{\beta_{m+1}}_{2}(m+1,t)&=-\lambda\left(I_t^{\beta_{m+1}}p^{\beta_{m+1}}_{1}(m+1,t)-I_t^{\beta_{m}}p^{\beta_{m}}_{1}(m,t)\right)=0.
	\end{align*}
	Let
	\begin{equation*}
	p^{\beta_{m+1}}_{k-1}(m+1,t)=0,\ \ k-1<m+1.
	\end{equation*}
	Then
	\begin{equation*}
	p^{\beta_{m+1}}_{k}(m+1,t)=-\lambda\left(I_t^{\beta_{m+1}}p^{\beta_{m+1}}_{k-1}(m+1,t)-I_t^{\beta_{m}}p^{\beta_{m}}_{k-1}(m,t)\right)=0,\ \ k<m+1.
	\end{equation*}
	Now for $k\geq m+1$, we have
	\begin{align*}
	p^{\beta_{m+1}}_{m+1}(m+1,t)&=-\lambda\left(I_t^{\beta_{m+1}}p^{\beta_{m+1}}_{m}(m+1,t)-I_t^{\beta_{m}}p^{\beta_{m}}_{m}(m,t)\right)\\
	&=\lambda\underset{\Omega^{m}_{m}}{\sum}\frac{\lambda^mI_t^{\beta_{m}}t^{\sum_{j=0}^mk_j\beta_j}}{\Gamma\left(\sum_{j=0}^mk_j\beta_j+1\right)}\\
	&=\lambda^{m+1} \underset{\Omega^{m+1}_{m+1}}{\sum}\frac{t^{\sum_{j=0}^{m+1}k_j\beta_j}}{\Gamma\left(\sum_{j=0}^{m+1}k_j\beta_j+1\right)},\\
	p^{\beta_{m+1}}_{m+2}(m+1,t)&=-\lambda\left(I_t^{\beta_{m+1}}p^{\beta_{m+1}}_{m+1}(m+1,t)-I_t^{\beta_{m}}p^{\beta_{m}}_{m+1}(m,t)\right)\\
	&=-\lambda\left(\underset{\Omega^{m+1}_{m+1}}{\sum}\frac{\lambda^{m+1}t^{\sum_{j=0}^{m+1}k_j\beta_j+\beta_{m+1}}}{\Gamma\left(\sum_{j=0}^{m+1}k_j\beta_j+\beta_{m+1}+1\right)}+\underset{\Omega^{m+1}_{m}}{\sum}\frac{\lambda^{m+1}t^{\sum_{j=0}^mk_j\beta_j+\beta_{m}}}{\Gamma\left(\sum_{j=0}^mk_j\beta_j+\beta_{m}+1\right)}\right)\\
	&=-\lambda^{m+2}\underset{\Omega^{m+2}_{m+1}}{\sum}\frac{t^{\sum_{j=0}^{m+1}k_j\beta_j}}{\Gamma\left(\sum_{j=0}^{m+1}k_j\beta_j+1\right)}.
	\end{align*}
	Let
	\begin{equation*}
	p^{\beta_{m+1}}_{k-1}(m+1,t)=(-1)^{m+1}(-\lambda)^{k-1}\underset{\Omega^{k-1}_{m+1}}{\sum}\frac{t^{\sum_{j=0}^{m+1}k_j\beta_j}}{\Gamma\left(\sum_{j=0}^{m+1}k_j\beta_j+1\right)},\ \ k-1\geq m+1.
	\end{equation*}
	Then
	\begin{align*}
	p^{\beta_{m+1}}_{k}(m+1,t)&=-\lambda\left(I_t^{\beta_{m+1}}p^{\beta_{m+1}}_{k-1}(m+1,t)-I_t^{\beta_{m}}p^{\beta_{m}}_{k-1}(m,t)\right)\\
	&=-\lambda\left((-1)^{m+1}\underset{\Omega^{k-1}_{m+1}}{\sum}\frac{(-\lambda)^{k-1}t^{\sum_{j=0}^{m+1}k_j\beta_j+\beta_{m+1}}}{\Gamma\left(\sum_{j=0}^{m+1}k_j\beta_j+\beta_{m+1}+1\right)}\right.\\
	&\ \ \ \ \ \ \ \ \ \ \ \ \ \ \ \ \ \ \ \ \ \ \ \ \left.+(-1)^{m+1}\underset{\Omega^{k-1}_{m}}{\sum}\frac{(-\lambda)^{k-1}t^{\sum_{j=0}^mk_j\beta_j+\beta_{m}}}{\Gamma\left(\sum_{j=0}^mk_j\beta_j+\beta_{m}+1\right)}\right)\\
	&=(-1)^{m+1}(-\lambda)^k\underset{\Omega^{k}_{m+1}}{\sum}\frac{t^{\sum_{j=0}^{m+1}k_j\beta_j}}{\Gamma\left(\sum_{j=0}^{m+1}k_j\beta_j+1\right)},\ \ k\geq m+1.
	\end{align*}
	Therefore,
	\begin{align*}
	p^{\beta_{m+1}}(m+1,t)&=(-1)^{m+1}\sum_{k=m+1}^{\infty}(-\lambda)^k\underset{\Omega^{k}_{m+1}}{\sum}\frac{t^{\sum_{j=0}^{m+1}k_j\beta_j}}{\Gamma\left(\sum_{j=0}^{m+1}k_j\beta_j+1\right)},
	\end{align*}
	and thus the result holds for $n=m+1$. This completes the proof.
\end{proof}
\begin{remark}
An alternate expression for (\ref{wq2.4kky}) is
\begin{equation}\label{teaa}
p^{\beta_n}(n,t)=\lambda^nt^{\sum_{j=0}^{n-1}\beta_j}\sum_{k=0}^{\infty‎}(-\lambda)^k\underset{k_j\in\mathbb{N}_0}{\underset{\sum_{j=0}^nk_j=k}{\sum}}\frac{t^{\sum_{j=0}^nk_j\beta_j}}{\Gamma(1+k_n\beta_n+\sum_{j=0}^{n-1}(k_j+1)\beta_j)},\ \ n\geq0.
\end{equation}
\end{remark}
	From the above result the Laplace transform of the state probabilities of SDTFPP-II is obtained in the following form:
   \begin{align*}
	\tilde{p}^{\beta_n}(n,s)&=\int_0^{\infty}p^{\beta_n}(n,t)e^{-st}\,\mathrm{d}t\\
	&=(-1)^n\sum_{k=n}^{\infty‎}\underset{\Omega^k_n}{\sum}\frac{(-\lambda)^k}{s^{k_0\beta_0+k_1\beta_1+\cdots+k_n\beta_n+1}},\ \ \ \ \ \ \ \mathrm{(using}\ \mathrm{(}\ref{wq2.4kky}\mathrm{))}\\
	&=\frac{\lambda^n}{s^{1+\sum_{j=0}^{n-1}\beta_j}}\sum_{k=n}^{\infty‎}\underset{k_n\in\mathbb{N}_0,\ k_j\in\mathbb{N}_0\backslash\{0\},\ 0\leq j\leq n-1}{\underset{k_0+k_1+\cdots+k_n=k}{\sum}}\frac{(-\lambda)^{(k_0-1)+(k_1-1)+\cdots+(k_{n-1}-1)+k_n}}{s^{(k_0-1)\beta_0+(k_1-1)\beta_1+\cdots+(k_{n-1}-1)\beta_{n-1}+k_n\beta_n}}\\
	&=\frac{\lambda^ns^{\beta_n-1}}{s^{\sum_{j=0}^n\beta_j}}\sum_{k=0}^{\infty‎}\underset{l_j\in\mathbb{N}_0,\ 0\leq j\leq n}{\underset{l_0+l_1+\cdots+l_n=k}{\sum}}\frac{(-\lambda)^{l_0+l_1+\cdots+l_n}}{s^{l_0\beta_0+l_1\beta_1+\cdots+l_n\beta_n}},\ \ (l_n=k_n,\ l_j=k_j-1,\ j\neq n)\\
	&=\frac{\lambda^ns^{\beta_n-1}}{s^{\beta_0+\beta_1+\cdots+\beta_n}}\prod_{k=0}^{n}\sum_{l_k=0}^{\infty‎}\left(\frac{-\lambda}{s^{\beta_k}}\right)^{l_k}\\
	&=\frac{\lambda^ns^{\beta_n-1}}{s^{\beta_0+\beta_1+\cdots+\beta_n}}\prod_{k=0}^{n}\left(1+\frac{\lambda}{s^{\beta_k}}\right)^{-1}\\
	&=\frac{\lambda^ns^{\beta_n-1}}{\prod_{k=0}^n(s^{\beta_k}+\lambda)},
	\end{align*}
	which coincides with Theorem 3.1 of Garra {\it et al.} (2015).
\begin{corollary}
	Let $X_2$ be the first waiting time of SDTFPP-II. Then the distribution of $X_2$ is given by:
	\begin{equation}
	\mathrm{Pr}\{X_2>t\}=\mathrm{Pr}\{N_2(t,\lambda)=0\}=E_{\beta_0}(-\lambda t^{\beta_0}).
	\end{equation}
\end{corollary}
\begin{remark}
The distribution of TFPP can be obtained as a special case of Theorem \ref{t2} by substituting $\beta_n=\beta$ for all $n\geq0$ in (\ref{wq2.4kky}) (or in (\ref{teaa})) and observing that the cardinality of $\Omega^k_n$ is same as that of the set $\Theta^k_n$, {\it i.e.}, $\binom{k}{n}$. Further, the substitution $\beta=1$ gives the distribution of the classical homogeneous Poisson process.
\end{remark}
\begin{remark}
From the distributions of SDTFPP-I and SDTFPP-II, it is clear that these are two distinct processes as observed by Garra {\it et al.} (2015) using the respective Laplace transforms. Moreover, from (\ref{2.4kky}) and (\ref{wq2.4kky}), the following relationship between the state probabilities of SDTFPP-I and SDTFPP-II (see Eq. (3.4), Garra {\it et al.} (2015)) can be easily verified:
\begin{equation*}
p^{\alpha_n}(n,t)=\left\{
\begin{array}{ll}
I^{\alpha_n-\alpha_0}_t\mathrm{Pr}\{N_2(t,\lambda)=n\},\ \ \alpha_n-\alpha_0>0,\\\\
D^{\alpha_0-\alpha_n}_t\mathrm{Pr}\{N_2(t,\lambda)=n\},\ \ \alpha_n-\alpha_0<0,
\end{array}
\right.
\end{equation*}
where $D^\alpha_t$ denotes the Riemann-Liouville (RL) fractional derivative which is defined by $D^\alpha_tf(t):=\frac{\mathrm{d}}{\mathrm{dt}}I^{1-\alpha}_tf(t)$, $0<\alpha<1$.
\end{remark}

\section{State dependent fractional pure birth process}
Recently, a state dependent version of FPBP, namely SDFPBP, was introduced and studied by Garra {\it et al.} (2015). They obtained the Laplace transform of the state probabilities of SDFPBP. Here we obtain the explicit expressions for state probabilities of SDFPBP using ADM.

\begin{theorem}\label{tq2}
	Consider the following difference-differential equations governing the state probabilities of SDFPBP:
	\begin{equation}\label{qr2new}
	\partial_t^{\nu_n} p^{\nu_n}(n,t)=-\lambda_np^{\nu_n}(n,t)+\lambda_{n-1}p^{\nu_{n-1}}(n-1,t),\ \ 0<\nu_n\leq 1,\ \lambda_n>0,\ n\geq 1,
	\end{equation}
	with $p^{\nu_1}(1,0)=1$ and $p^{\nu_n}(n,0)=0$, $n\geq 2$. The solution of (\ref{qr2new}) is given by
	\begin{equation}\label{q2.4kky}
	p^{\nu_n}(n,t)=(-1)^{n-1}\frac{\lambda_1}{\lambda_n}\sum_{k=n-1}^{\infty‎}(-1)^k\underset{\Lambda^k_n}{\sum}\frac{t^{\sum_{j=1}^nk_j\nu_j}\prod_{j=1}^n\lambda_j^{k_j}}{\Gamma(1+\sum_{j=1}^nk_j\nu_j)},\ \ n\geq1,
	\end{equation}
	where 
	$
	\Lambda^k_n=\{(k_1,k_2,\ldots,k_n):\ \sum_{j=1}^nk_j=k,\ k_1\in\mathbb{N}_0,\ k_j\in\mathbb{N}_0\backslash\{0\},\ 2\leq j\leq n\}.
	$
\end{theorem}
\begin{proof}
	Applying RL integral $I^{\nu_n}_t$ on both sides of (\ref{qr2new}), we get
	\begin{equation}\label{qyth}
	p^{\nu_n}(n,t)=p^{\nu_n}(n,0)+I_t^{\nu_n}(-\lambda_np^{\nu_n}(n,t)+\lambda_{n-1}p^{\nu_{n-1}}(n-1,t)),\ \ n\geq 1.
	\end{equation}
	Note that $p^{\nu_{0}}(0,t)=0$ for $t\geq0$.
	
	For $n=1$, substitute $p^{\nu_1}(1,t)=\sum_{k=0}^{\infty}p^{\nu_1}_{k}(1,t)$ in (\ref{qyth}) and apply ADM to get
	\begin{equation*}
	\sum_{k=0}^{\infty}p^{\nu_1}_{k}(1,t)=p^{\nu_1}(1,0)-\lambda_1\sum_{k=0}^{\infty} I_t^{\nu_1} p^{\nu_1}_{k}(1,t).
	\end{equation*}
	Thus, $p^{\nu_1}_{0}(1,t)=p^{\nu_1}(1,0)=1$ and $p^{\nu_1}_{k}(1,t)=-\lambda_1 I_t^{\nu_1} p^{\nu_1}_{k-1}(1,t)$, $k\geq 1$. Hence,
	\begin{equation*}
	p^{\nu_1}_{1}(1,t)=-\lambda_1 I_t^{\nu_1} p^{\nu_1}_{0}(1,t)=-\lambda_1 I_t^{\nu_1} t^0=\frac{-\lambda_1 t^{\nu_1}}{\Gamma(\nu_1+1)},
	\end{equation*}
	and similarly 
	\begin{equation*}
	p^{\nu_1}_{2}(1,t)=\frac{(-\lambda_1 t^{\nu_1})^2}{\Gamma(2\nu_1+1)},\ \ p^{\nu_1}_{3}(1,t)=\frac{(-\lambda_1 t^{\nu_1})^3}{\Gamma(3\nu_1+1)}.
	\end{equation*}
	Let
	\begin{equation}
	p^{\nu_1}_{k-1}(1,t)=\frac{(-\lambda_1 t^{\nu_1})^{k-1}}{\Gamma((k-1){\nu_1}+1)}.
	\end{equation}
	Then
	\begin{equation*}
	p^{\nu_1}_{k}(1,t)=-\lambda_1 I_t^{\nu_1} p^{\nu_1}_{k-1}(1,t)=\frac{(-\lambda_1)^{k}}{\Gamma((k-1){\nu_1}+1)} I_t^{\nu_1} t^{(k-1){\nu_1}}=\frac{(-\lambda_1 t^{\nu_1})^{k}}{\Gamma(k\nu_1+1)},\ \ k\geq 0.
	\end{equation*}
	Therefore,
	\begin{equation}
	p^{\nu_1}(1,t)=\sum_{k=0}^{\infty}\frac{(-\lambda_1 t^{\nu_1})^{k}}{\Gamma(k{\nu_1}+1)},
	\end{equation}
	and thus the result holds for $n=1$.
	
	For $n=2$, substituting $p^{\nu_2}(2,t)=\sum_{k=0}^{\infty}p^{\nu_2}_{k}(2,t)$ in (\ref{qyth}) and applying ADM, we get
	\begin{equation*}
	\sum_{k=0}^{\infty}p^{\nu_2}_{k}(2,t)=p^{\nu_2}(2,0)+\sum_{k=0}^{\infty} I_t^{\nu_2} \left(-\lambda_{2}p^{\nu_2}_{k}(2,t)+\lambda_{1}p^{\nu_1}_{k}(1,t)\right).
	\end{equation*}
	Thus, $p^{\nu_2}_{0}(2,t)=p^{\nu_2}(2,0)=0$ and $p^{\nu_2}_{k}(2,t)=I_t^{\nu_2} \left(-\lambda_{2}p^{\nu_2}_{k-1}(2,t)+\lambda_{1}p^{\nu_1}_{k-1}(1,t)\right)$, $k\geq 1$.\\
	Hence,
	\begin{align*}
	p^{\nu_2}_{1}(2,t)&=I_t^{\nu_2} \left(-\lambda_{2}p^{\nu_2}_{0}(2,t)+\lambda_{1}p^{\nu_1}_{0}(1,t)\right)=\lambda_1 I_t^{\nu_2} t^0=\frac{\lambda_1t^{\nu_2}}{\Gamma({\nu_2}+1)},\\
	p^{\nu_2}_{2}(2,t)&=I_t^{\nu_2} \left(-\lambda_{2}p^{\nu_2}_{1}(2,t)+\lambda_{1}p^{\nu_1}_{1}(1,t)\right)\\
	&=I_t^{\nu_2}\left(-\frac{\lambda_1\lambda_2 t^{\nu_2}}{\Gamma(\nu_2+1)}-\frac{\lambda_1^2 t^{\nu_1}}{\Gamma(\nu_1+1)}\right)=-\frac{\lambda_1\lambda_2t^{2\nu_2}}{\Gamma(2\nu_2+1)}-\frac{\lambda_1^2t^{\nu_1+\nu_2}}{\Gamma(\nu_1+\nu_2+1)},\\
	p^{\nu_2}_{3}(2,t)&=I_t^{\nu_2} \left(-\lambda_{2}p^{\nu_2}_{2}(2,t)+\lambda_{1}p^{\nu_1}_{2}(1,t)\right)\\
	&=\frac{\lambda_1\lambda_2^2t^{3\nu_2}}{\Gamma(3\nu_2+1)}+\frac{\lambda_1^2\lambda_2 t^{\nu_1+2\nu_2}}{\Gamma(\nu_1+2\nu_2+1)}+\frac{\lambda_1^3 t^{2\nu_1+\nu_2}}{\Gamma(2\nu_1+\nu_2+1)}.
	\end{align*}
	Let us assume now
	\begin{equation}
	p^{\nu_2}_{k-1}(2,t)=-(-1)^{k-1}\underset{\Lambda^{k-1}_{2}}{\sum}\frac{\lambda_1^{k_1+1}\lambda_2^{k_2-1}t^{k_1\nu_1+k_2\nu_2}}{\Gamma\left(k_1\nu_1+k_2\nu_2+1\right)}.
	\end{equation}
	Then
	\begin{align*}
	p^{\nu_2}_{k}(2,t)&=I_t^{\nu_2} \left(-\lambda_{2}p^{\nu_2}_{k-1}(2,t)+\lambda_{1}p^{\nu_1}_{k-1}(1,t)\right)\\
	&=(-1)^{k-1} \left(\underset{\Lambda^{k-1}_{2}}{\sum}\frac{\lambda_1^{k_1+1}\lambda_2^{k_2}t^{k_1\nu_1+k_2\nu_2+\nu_2}}{\Gamma\left(k_1\nu_1+k_2\nu_2+\nu_2+1\right)}+\frac{\lambda_1^kt^{(k-1)\nu_1+\nu_2}}{\Gamma((k-1){\nu_1}+\nu_2+1)}\right)\\
	&=-(-1)^k\underset{\Lambda^{k}_{2}}{\sum}\frac{\lambda_1^{k_1+1}\lambda_2^{k_2-1}t^{k_1\nu_1+k_2\nu_2}}{\Gamma\left(k_1\nu_1+k_2\nu_2+1\right)},\ \ k\geq 1.
	\end{align*}
	Therefore,
	\begin{equation}
	p^{\nu_2}(2,t)=-\frac{\lambda_1}{\lambda_2}\sum_{k=1}^{\infty}(-1)^k\underset{\Lambda^{k}_{2}}{\sum}\frac{\lambda_1^{k_1}\lambda_2^{k_2}t^{k_1\nu_1+k_2\nu_2}}{\Gamma\left(k_1\nu_1+k_2\nu_2+1\right)},
	\end{equation}
	and thus the result holds for $n=2$.
	
	For $n=3$, substituting $p^{\nu_3}(3,t)=\sum_{k=0}^{\infty}p^{\nu_3}_{k}(3,t)$ in (\ref{qyth}) and applying ADM, we get
	\begin{equation*}
	\sum_{k=0}^{\infty}p^{\nu_3}_{k}(3,t)=p^{\nu_3}(3,0)+\sum_{k=0}^{\infty} I_t^{\nu_3} \left(-\lambda_3p^{\nu_3}_{k}(3,t)+\lambda_2p^{\nu_2}_{k}(2,t)\right).
	\end{equation*}
	Thus, $p^{\nu_3}_{0}(3,t)=p^{\nu_3}(3,0)=0$ and $p^{\nu_3}_{k}(3,t)=I_t^{\nu_3} \left(-\lambda_3p^{\nu_3}_{k-1}(3,t)+\lambda_2p^{\nu_2}_{k-1}(2,t)\right)$, $k\geq 1$.\\
	Hence,
	\begin{align*}
	p^{\nu_3}_{1}(3,t)&=I_t^{\nu_3} \left(-\lambda_3p^{\nu_3}_{0}(3,t)+\lambda_2p^{\nu_2}_{0}(2,t)\right)=0,\\
	p^{\nu_3}_{2}(3,t)&=I_t^{\nu_3} \left(-\lambda_3p^{\nu_3}_{1}(3,t)+\lambda_2p^{\nu_2}_{1}(2,t)\right)=I_t^{\nu_3}\left(\frac{\lambda_1\lambda_2 t^{\nu_2}}{\Gamma({\nu_2}+1)}\right)=\frac{\lambda_1\lambda_2t^{\nu_2+\nu_3}}{\Gamma(\nu_2+\nu_3+1)},\\
	p^{\nu_3}_{3}(3,t)&=I_t^{\nu_3} \left(-\lambda_3p^{\nu_3}_{2}(3,t)+\lambda_2p^{\nu_2}_{2}(2,t)\right)\\
	&=-\frac{\lambda_1\lambda_2\lambda_3t^{\nu_2+2\nu_3}}{\Gamma(\nu_2+2\nu_3+1)}-\frac{\lambda_1\lambda_2^2t^{2\nu_2+\nu_3}}{\Gamma(2\nu_2+\nu_3+1)}-\frac{\lambda_1^2\lambda_2t^{\nu_1+\nu_2+\nu_3}}{\Gamma(\nu_1+\nu_2+\nu_3+1)}.
	\end{align*}
	Let
	\begin{equation}
	p^{\nu_3}_{k-1}(3,t)=(-1)^{k-1}\underset{\Lambda^{k-1}_{3}}{\sum}\frac{\lambda_1^{k_1+1}\lambda_2^{k_2}\lambda_3^{k_3-1}t^{k_1\nu_1+k_2\nu_2+k_3\nu_3}}{\Gamma\left(k_1\nu_1+k_2\nu_2+k_3\nu_3+1\right)}.
	\end{equation}
	Then
	\begin{align*}
	p^{\nu_3}_{k}(3,t)&=I_t^{\nu_3} \left(-\lambda_3p^{\nu_3}_{k-1}(3,t)+\lambda_2p^{\nu_2}_{k-1}(2,t)\right)\\
	&=(-1)^{k}\left(\underset{\Lambda^{k-1}_{3}}{\sum}\frac{\lambda_1^{k_1+1}\lambda_2^{k_2}\lambda_3^{k_3}t^{k_1\nu_1+k_2\nu_2+k_3\nu_3+\nu_3}}{\Gamma\left(k_1\nu_1+k_2\nu_2+k_3\nu_3+\nu_3+1\right)}+\underset{\Lambda^{k-1}_{2}}{\sum}\frac{\lambda_1^{k_1+1}\lambda_2^{k_2}t^{k_1\nu_1+k_2\nu_2+\nu_3}}{\Gamma\left(k_1\nu_1+k_2\nu_2+\nu_3+1\right)}\right)\\
	&=(-1)^k\underset{\Lambda^{k}_{3}}{\sum}\frac{\lambda_1^{k_1+1}\lambda_2^{k_2}\lambda_3^{k_3-1}t^{k_1\nu_1+k_2\nu_2+k_3\nu_3}}{\Gamma\left(k_1\nu_1+k_2\nu_2+k_3\nu_3+1\right)},\ \ k\geq 2.
	\end{align*}
	Therefore,
	\begin{equation}
	p^{\nu_3}(3,t)=\frac{\lambda_1}{\lambda_3}\sum_{k=2}^{\infty}(-1)^k\underset{\Lambda^{k}_{3}}{\sum}\frac{\lambda_1^{k_1}\lambda_2^{k_2}\lambda_3^{k_3}t^{k_1\nu_1+k_2\nu_2+k_3\nu_3}}{\Gamma\left(k_1\nu_1+k_2\nu_1+k_3\nu_3+1\right)},
	\end{equation}
	and thus the result holds for $n=3$.
	
	Let $p^{\nu_m}(m,t)=\sum_{k=0}^{\infty}p^{\nu_m}_{k}(m,t)$ in (\ref{qyth}) and assume the result holds for $n=m>3$, {\it i.e.}, $p^{\nu_m}_{k}(m,t)=0$, $k<m-1$ and
	\begin{equation*}
	p^{\nu_m}_{k}(m,t)=(-1)^{m-1}(-1)^k\underset{\Lambda^{k}_{m}}{\sum}\frac{\lambda_1^{k_1+1}\lambda_2^{k_2}\ldots\lambda_{m-1}^{k_{m-1}}\lambda_m^{k_m-1}t^{\sum_{j=1}^mk_j\nu_j}}{\Gamma\left(\sum_{j=1}^mk_j\nu_j+1\right)},\ \ k\geq m-1.
	\end{equation*}
	For $n=m+1$, substituting $p^{\nu_{m+1}}(m+1,t)=\sum_{k=0}^{\infty}p^{\nu_{m+1}}_{k}(m+1,t)$ in (\ref{qyth}) and applying ADM, we get
	\begin{equation*}
	\sum_{k=0}^{\infty}p^{\nu_{m+1}}_{k}(m+1,t)=p^{\nu_{m+1}}(m+1,0)+\sum_{k=0}^{\infty} I_t^{\nu_{m+1}} \left(-\lambda_{m+1}p^{\nu_{m+1}}_{k}(m+1,t)+\lambda_{m}p^{\nu_{m}}_{k}(m,t)\right).
	\end{equation*}
	Thus, $p^{\nu_{m+1}}_{0}(m+1,t)=p^{\nu_{m+1}}(m+1,0)=0$ and\\ $p^{\nu_{m+1}}_{k}(m+1,t)=I_t^{\nu_{m+1}} \left(-\lambda_{m+1}p^{\nu_{m+1}}_{k-1}(m+1,t)+\lambda_{m}p^{\nu_{m}}_{k-1}(m,t)\right)$, $k\geq 1$. Hence,
	\begin{align*}
	p^{\nu_{m+1}}_{1}(m+1,t)&=I_t^{\nu_{m+1}} \left(-\lambda_{m+1}p^{\nu_{m+1}}_{0}(m+1,t)+\lambda_{m}p^{\nu_{m}}_{0}(m,t)\right)=0,\\
	p^{\nu_{m+1}}_{2}(m+1,t)&=I_t^{\nu_{m+1}} \left(-\lambda_{m+1}p^{\nu_{m+1}}_{1}(m+1,t)+\lambda_{m}p^{\nu_{m}}_{1}(m,t)\right)=0.
	\end{align*}
	Let
	\begin{equation*}
	p^{\nu_{m+1}}_{k-1}(m+1,t)=0,\ \ k-1<m.
	\end{equation*}
	Then
	\begin{equation*}
	p^{\nu_{m+1}}_{k}(m+1,t)=I_t^{\nu_{m+1}} \left(-\lambda_{m+1}p^{\nu_{m+1}}_{k-1}(m+1,t)+\lambda_{m}p^{\nu_{m}}_{k-1}(m,t)\right)=0,\ \ k<m.
	\end{equation*}
	Now for $k\geq m$, we have
	\begin{align*}
	p^{\nu_{m+1}}_{m}(m+1,t)&=I_t^{\nu_{m+1}} \left(-\lambda_{m+1}p^{\nu_{m+1}}_{m-1}(m+1,t)+\lambda_{m}p^{\nu_{m}}_{m-1}(m,t)\right)\\
	&=I_t^{\nu_{m+1}} \left(\underset{\Lambda^{m-1}_{m}}{\sum}\frac{\lambda_1^{k_1+1}\lambda_2^{k_2}\ldots\lambda_{m-1}^{k_{m-1}}\lambda_m^{k_m}t^{\sum_{j=1}^mk_j\nu_j}}{\Gamma\left(\sum_{j=1}^mk_j\nu_j+1\right)}\right)\\
	&=\underset{\Lambda^{m-1}_{m}}{\sum}\frac{\lambda_1^{k_1+1}\lambda_2^{k_2}\ldots\lambda_{m-1}^{k_{m-1}}\lambda_m^{k_m}t^{\sum_{j=1}^mk_j\nu_j+\nu_{m+1}}}{\Gamma\left(\sum_{j=1}^mk_j\nu_j+\nu_{m+1}+1\right)}\\
	&= \underset{\Lambda^{m}_{m+1}}{\sum}\frac{\lambda_1^{k_1+1}\lambda_2^{k_2}\ldots\lambda_{m}^{k_{m}}\lambda_{m+1}^{k_{m+1}-1}t^{\sum_{j=1}^{m+1}k_j\nu_j}}{\Gamma\left(\sum_{j=1}^{m+1}k_j\nu_j+1\right)}.
    \end{align*}
Consider next
	\begin{align*}
	p^{\nu_{m+1}}_{m+1}(m+1,t)&=I_t^{\nu_{m+1}} \left(-\lambda_{m+1}p^{\nu_{m+1}}_{m}(m+1,t)+\lambda_{m}p^{\nu_{m}}_{m}(m,t)\right)\\
	&=-\underset{\Lambda^{m}_{m+1}}{\sum}\frac{\lambda_1^{k_1+1}\lambda_2^{k_2}\ldots\lambda_{m}^{k_{m}}\lambda_{m+1}^{k_{m+1}}t^{\sum_{j=1}^{m+1}k_j\nu_j+\nu_{m+1}}}{\Gamma\left(\sum_{j=1}^{m+1}k_j\nu_j+\nu_{m+1}+1\right)}\\
	&\ \ \ \ \ \ \ \ \ \ \ \ \ \ \ \ \ -\underset{\Lambda^{m}_{m}}{\sum}\frac{\lambda_1^{k_1+1}\lambda_2^{k_2}\ldots\lambda_{m-1}^{k_{m-1}}\lambda_{m}^{k_{m}}t^{\sum_{j=1}^mk_j\nu_j+\nu_{m+1}}}{\Gamma\left(\sum_{j=1}^mk_j\nu_j+\nu_{m+1}+1\right)}\\
	&=-\underset{\Lambda^{m+1}_{m+1}}{\sum}\frac{\lambda_1^{k_1+1}\lambda_2^{k_2}\ldots\lambda_{m}^{k_{m}}\lambda_{m+1}^{k_{m+1}-1}t^{\sum_{j=1}^{m+1}k_j\nu_j}}{\Gamma\left(\sum_{j=1}^{m+1}k_j\nu_j+1\right)}.
	\end{align*}
	Let now
	\begin{equation*}
	p^{\nu_{m+1}}_{k-1}(m+1,t)=(-1)^{m+k-1}\underset{\Lambda^{k-1}_{m+1}}{\sum}\frac{\lambda_1^{k_1+1}\lambda_2^{k_2}\ldots\lambda_{m}^{k_{m}}\lambda_{m+1}^{k_{m+1}-1}t^{\sum_{j=1}^{m+1}k_j\nu_j}}{\Gamma\left(\sum_{j=1}^{m+1}k_j\nu_j+1\right)},\ \ k-1\geq m.
	\end{equation*}
	Then
	\begin{align*}
	p^{\nu_{m+1}}_{k}(m+1,t)&= I_t^{\nu_{m+1}} \left(-\lambda_{m+1}p^{\nu_{m+1}}_{k-1}(m+1,t)+\lambda_mp^{\nu_{m}}_{k-1}(m,t)\right)\\
	&=(-1)^{m}(-1)^{k}\left(\underset{\Lambda^{k-1}_{m+1}}{\sum}\frac{\lambda_1^{k_1+1}\lambda_2^{k_2}\ldots\lambda_{m}^{k_{m}}\lambda_{m+1}^{k_{m+1}}t^{\sum_{j=1}^{m+1}k_j\nu_j+\nu_{m+1}}}{\Gamma\left(\sum_{j=1}^{m+1}k_j\nu_j+\nu_{m+1}+1\right)}\right.\\
	&\ \ \ \ \ \ \ \ \ \ \ \ \ \ \ \ \ \ \ \ \ \ \ \ \ \ \ \ \left.+\underset{\Lambda^{k-1}_{m}}{\sum}\frac{\lambda_1^{k_1+1}\lambda_2^{k_2}\ldots\lambda_{m-1}^{k_{m-1}}\lambda_m^{k_m}t^{\sum_{j=1}^mk_j\nu_j+\nu_{m+1}}}{\Gamma\left(\sum_{j=1}^mk_j\nu_j+\nu_{m+1}+1\right)}\right)\\
	&=(-1)^{m}(-1)^k\underset{\Lambda^{k}_{m+1}}{\sum}\frac{\lambda_1^{k_1+1}\lambda_2^{k_2}\ldots\lambda_{m}^{k_{m}}\lambda_{m+1}^{k_{m+1}-1}t^{\sum_{j=1}^{m+1}k_j\nu_j}}{\Gamma\left(\sum_{j=1}^{m+1}k_j\nu_j+1\right)},\ \ k\geq m.
	\end{align*}
	Therefore,
	\begin{align*}
	p^{\nu_{m+1}}(m+1,t)&=(-1)^{m}\frac{\lambda_1}{\lambda_{m+1}}\sum_{k=m}^{\infty}(-1)^k\underset{\Lambda^{k}_{m+1}}{\sum}\frac{\lambda_1^{k_1}\lambda_2^{k_2}\ldots\lambda_{m+1}^{k_{m+1}}t^{\sum_{j=1}^{m+1}k_j\nu_j}}{\Gamma\left(\sum_{j=1}^{m+1}k_j\nu_j+1\right)},
	\end{align*}
	and thus the result holds for $n=m+1$. This completes the proof.
\end{proof}
\begin{remark}
The following alternate form of the state probabilities of SDFPBP can be obtained from (\ref{q2.4kky}):
\begin{equation*}
p^{\nu_n}(n,t)=\left(\prod_{j=1}^{n-1}\lambda_j^{k_j}\right)t^{\sum_{j=2}^n\nu_j}\sum_{k=0}^{\infty‎}(-1)^k\underset{k_j\in\mathbb{N}_0}{\underset{\sum_{j=0}^nk_j=k}{\sum}}\frac{t^{\sum_{j=1}^nk_j\nu_j}\prod_{j=1}^n\lambda_j^{k_j}}{\Gamma(1+k_1\nu_1+\sum_{j=2}^n(k_j+1)\nu_j)},\ \ n\geq0.
\end{equation*}
\end{remark}
The Laplace transform of the state probabilities of SDFPBP can be obtained from (\ref{q2.4kky}), {\it viz.},
	\begin{align*}
	\tilde{p}^{\nu_n}(n,s)&=\int_0^{\infty}p^{\nu_n}(n,t)e^{-st}\,\mathrm{d}t\\
	&=(-1)^{n-1}\frac{\lambda_1}{\lambda_n}\sum_{k=n-1}^{\infty‎}(-1)^k\underset{\Lambda^k_n}{\sum}\frac{\lambda_1^{k_1}\lambda_2^{k_2}\ldots\lambda_n^{k_n}}{s^{k_1\nu_1+k_2\nu_2+\cdots+k_n\nu_n+1}}\\
	&=\frac{\prod_{k=1}^{n-1}\lambda_k}{s^{\nu_2+\nu_3+\cdots+\nu_n+1}}\sum_{k=n-1}^{\infty‎}\underset{k_1\in\mathbb{N}_0,\ k_j\in\mathbb{N}_0\backslash\{0\},\ 2\leq j\leq n}{\underset{k_1+k_2+\cdots+k_n=k}{\sum}}\frac{(-\lambda_1)^{k_1}(-\lambda_2)^{k_2-1}\ldots(-\lambda_n)^{k_n-1}}{s^{k_1\nu_1+(k_2-1)\nu_2+\cdots+(k_n-1)\nu_n}}\\
	&=\frac{s^{\nu_1-1}\prod_{k=1}^{n-1}\lambda_k}{s^{\nu_1+\nu_2+\cdots+\nu_n}}\sum_{k=0}^{\infty‎}\underset{l_j\in\mathbb{N}_0,\ 1\leq j\leq n}{\underset{l_1+l_2+\cdots+l_n=k}{\sum}}\prod_{k=1}^{n}\left(\frac{-\lambda_k}{s^{\nu_k}}\right)^{l_k},\ \ (l_1=k_1,\ l_j=k_j-1,\ j\neq 1)\\
	&=\frac{s^{\nu_1-1}\prod_{k=1}^{n-1}\lambda_k}{s^{\nu_1+\nu_2+\cdots+\nu_n}}\prod_{k=1}^{n}\sum_{l_k=0}^{\infty‎}\left(\frac{-\lambda_k}{s^{\nu_k}}\right)^{l_k}\\
	&=\frac{s^{\nu_1-1}\prod_{k=1}^{n-1}\lambda_k}{s^{\nu_1+\nu_2+\cdots+\nu_n}}\prod_{k=1}^{n}\left(1+\frac{\lambda_k}{s^{\nu_k}}\right)^{-1}\\
	&=\frac{s^{\nu_1-1}\prod_{k=1}^{n-1}\lambda_k}{\prod_{k=1}^n(s^{\nu_k}+\lambda_k)},
	\end{align*}
	which coincides with Proposition 4.1 of Garra {\it et al.} (2015).
\begin{corollary}\label{cdbe}
	Let $Y_1,Y_2,\ldots,Y_n$ be $n$ independent random variables such that $Y_j$ follows Mittag-Leffler distribution (see Cahoy {\it et al.} (2010)) with distribution function $F_{Y_j}(t)=1-E_{\nu_j}(-\lambda_jt^{\nu_j})$, where $0<\nu_j\leq1$, $1\leq j\leq n$ and $\nu_1=\lambda_n=1$. Then, the density function of the convolution $W=Y_1+Y_2+\cdots+Y_n$ is given by
\begin{equation*}
f_W(t)=(-1)^{n-1}\lambda_1\sum_{k=n-1}^{\infty‎}(-1)^k\underset{\Lambda^k_n}{\sum}\frac{t^{k_1+\sum_{j=2}^nk_j\nu_j}\prod_{j=1}^{n-1}\lambda_j^{k_j}}{\Gamma(1+k_1+\sum_{j=2}^nk_j\nu_j)},\ \ t\geq0,
\end{equation*}
where $\Lambda^k_n=\{(k_1,k_2,\ldots,k_n):\ \sum_{j=1}^nk_j=k,\ k_1\in\mathbb{N}_0,\ k_j\in\mathbb{N}_0\backslash\{0\},\ 2\leq j\leq n\}.$
\end{corollary}
\begin{proof}
	The Laplace transform of the density of Mittag-Leffler random variables $Y_j$'s is
	\begin{equation*}
	\tilde{f}_{Y_j}(s)=\mathbb{E}(e^{-sX_j})=\frac{\lambda_j}{\lambda_j+s^{\alpha_j}},\ \ s>0.
	\end{equation*}
	Hence, the Laplace transform of the convolution is given by
	\begin{equation*}
	\tilde{f}_{W}(s)=\mathbb{E}(e^{-sW})=\prod_{j=1}^{n}\mathbb{E}(e^{-sY_j})=\frac{\prod_{j=1}^{n-1}\lambda_j}{\prod_{j=1}^n(\lambda_j+s^{\alpha_j})},\ \ s>0.
	\end{equation*}
	The proof follows in view of (\ref{lt3}) and (\ref{q2.4kky}). 
\end{proof}
The above result coincides with Corollary \ref{cdb} on substituting $\lambda_j=1$, $j\geq1$. For the next result we note that the first event of SDFPBP occurs at time $t=0$.
\begin{corollary}
	Let $\mathcal{X}$ denote the time of second event for SDFPBP. Then the distribution of $\mathcal{X}$ is given by:
	\begin{equation}
	\mathrm{Pr}\{\mathcal{X}>t\}=\mathrm{Pr}\{\mathcal{N}(t,\lambda)=1\}=E_{\nu_1}(-\lambda_1 t^{\nu_1}).
	\end{equation}
\end{corollary}
The distribution of state dependent linear birth process (SDLBP) can be obtained by substituting $\lambda_n=\lambda n$, $n\geq1$, in (\ref{q2.4kky}).
\begin{corollary}
	The pmf of SDLBP is given by
	\begin{equation*}
	p^{\nu_n}_*(n,t)=\frac{(-1)^{n-1}}{n}\sum_{k=n-1}^{\infty‎}(-\lambda)^k\underset{\Lambda^k_n}{\sum}\frac{t^{\sum_{j=1}^nk_j\nu_j}\prod_{j=1}^nj^{k_j}}{\Gamma(1+\sum_{j=1}^nk_j\nu_j)},\ \ n\geq1.
	\end{equation*}
\end{corollary}
The distribution of FPBP (see Orsingher and Polito (2010)) can be obtained by substituting $\nu_n=\nu$ for all $n\geq1$ in (\ref{q2.4kky}).
\begin{corollary}
	The pmf of FPBP is given by
	\begin{equation*}
	p^{\nu}(n,t)=(-1)^{n-1}\frac{\lambda_1}{\lambda_n}\sum_{k=n-1}^{\infty‎}\frac{(-t)^{k\nu}}{\Gamma\left(k\nu+1\right)}\underset{\Lambda^k_n}{\sum}\prod_{j=1}^n\lambda_j^{k_j},\ \ n\geq1.
	\end{equation*}
\end{corollary}

\section{Concluding remarks}
Recently, certain state dependent fractional point processes were considered by Garra {\it et al.} (2015). They obtained the Laplace transforms for these processes, but did not provide the explicit expressions of the state probabilities. This is because inversion of those Laplace transforms are difficult. In this paper, we obtain the explicit expressions for the state probabilities of these processes using Adomian decomposition method. From the distributions obtained, we evaluate the Laplace transform in each case and show that they coincide with the results obtained in Garra {\it et al.} (2015). Some convolutions of the Mittag-Leffler distributions are obtained as the particular cases of our results.
%\vspace*{0.5cm}
%
%\noindent {\bf Acknowledgements}
%A part of this work was done while the first author was visiting Institut des Hautes \'Etudes Scientifiques to attend a workshop, and the second author was visiting the Department of Statistics and Probability, Michigan State University, during Summer 2017. Also, the authors would like to thank the institutes for the hospitality and providing excellent research environment.

\end{document}